\documentclass[12pt]{article}
\usepackage[centertags]{amsmath}
\usepackage{amsfonts}
\usepackage{amssymb}
\usepackage{latexsym}
\usepackage{amsthm}
\usepackage{newlfont}
\usepackage{graphicx}
\usepackage{listings}
\usepackage{booktabs}
\usepackage{abstract}
\usepackage{enumerate}
\usepackage{xcolor}
\RequirePackage{srcltx}
\date{}

\newlength{\defbaselineskip}
\newcommand{\setlinespacing}[1]%
           {\setlength{\baselineskip}{#1 \defbaselineskip}}

\newcommand{\actaqed}{\hfill $\actabox$}
{\medskip\noindent \textit{Proof of #1. }}%
{\actaqed \medskip}

\def\cA{{\mathcal A}}

\def\cC{{\mathcal C}}
\def\cD{{\mathcal D}}

\def\cG{{\mathcal G}}

\def\cK{{\mathcal K}}

\def\cM{{\mathcal M}}

\def\cT{{\mathcal T}}

\def\cV{{\mathcal V}}

\def\cX{{\mathcal X}}

\def\bbC{{\mathbb C}}

\def\bbN{{\mathbb N}}

\def\bbR{{\mathbb R}}

\def\bbT{{\mathbb T}}

\def\bbZ{{\mathbb Z}}

\def\bA{{\mathbf A}}

\def\bF{{\mathbf F}}

\def\bH{{\mathbf H}}

\def\bN{{\mathbf N}}

\def\bW{{\mathbf W}}

\def\bj{\mathbf j}
\def\bk{\mathbf k}

\def\bs{\mathbf s}
\def\bt{\mathbf t}
\def\btt{\mathbf t}

\def\bw{\mathbf w}
\def\bx{\mathbf x}
\def\by{\mathbf y}

 \def \<{\langle}
\def\>{\rangle}

\def \Og{\Omega}
\def \ep{\epsilon}
\def \e{\varepsilon}
\def \va{\varepsilon}

\def \ff{\varphi}
\def\al{\alpha}
\def\bt{\beta}
\def \ro{\varrho}

\def\th{\theta}

\def\vi{\varphi}
\def \conv{\operatorname{conv}}
\newcommand\LD{\mathcal{LD}}

\def \sp{\operatorname{span}}

\def\bt{\beta}

\newtheorem{Theorem}{Theorem}[section]
\newtheorem{Lemma}{Lemma}[section]
\newtheorem{Definition}{Definition}[section]
\newtheorem{Proposition}{Proposition}[section]
\newtheorem{Remark}{Remark}[section]

\newtheorem{Corollary}{Corollary}[section]
\numberwithin{equation}{section}

\newcommand{\be}{\begin{equation}}
\newcommand{\ee}{\end{equation}}

\begin{document}

\title{A survey on sampling recovery}

\author{F. Dai  and   V. Temlyakov 	\footnote{
		This work was supported by the Russian Science Foundation under grant no. 23-71-30001, https://rscf.ru/project/23-71-30001/, and performed at Lomonosov Moscow State University.
  }}

\newcommand{\Addresses}{{
  \bigskip
  \footnotesize

  F.~Dai, \textsc{ Department of Mathematical and Statistical Sciences\\
University of Alberta\\ Edmonton, Alberta T6G 2N8, Canada\\
E-mail:} \texttt{fdai@ualberta.ca }

 \medskip
  V.N. Temlyakov, \textsc{ Steklov Mathematical Institute of Russian Academy of Sciences, Moscow, Russia;\\ Lomonosov Moscow State University; \\ Moscow Center of Fundamental and Applied Mathematics; \\ University of South Carolina.
  \\
E-mail:} \texttt{temlyakovv@gmail.com}

}}

\maketitle

\begin{abstract}{The reconstruction of unknown functions from a finite number of samples is a fundamental challenge in pure and applied mathematics. This survey provides a comprehensive overview of recent developments in sampling recovery, focusing on the accuracy of various algorithms and the relationship between optimal recovery errors, nonlinear approximation, and the Kolmogorov widths    of function classes. A central theme is the synergy between the theory of universal sampling discretization and Lebesgue-type inequalities for greedy algorithms. We discuss three primary algorithmic frameworks: weighted least squares and $\ell_p$ minimization, sparse approximation methods, and greedy algorithms such as the Weak Orthogonal Matching Pursuit (WOMP) in Hilbert spaces and the Weak Tchebychev Greedy Algorithm (WCGA) in Banach spaces. These methods are applied to function classes defined by structural conditions, like the $A_\beta^r$ and Wiener-type classes, as well as classical Sobolev-type classes with dominated mixed derivatives. Notably, we highlight recent findings showing that nonlinear sampling recovery can provide superior error guarantees compared to linear methods for certain multivariate function classes.
		}
\end{abstract}

\newpage

 \tableofcontents

 \newpage

\section{Introduction}
\label{RI}

The problem of recovery (reconstruction) of an unknown function defined on a subset of  $\bbR^d$ from its samples at a finite number of points is a fundamental problem of pure and applied mathematics. We would like to construct recovering operators (algorithms) which are good in the sense of accuracy, stability, and computational complexity. In this paper we discuss
the issue of accuracy. Following a standard  approach in approximation theory we define some
optimal characteristics -- the Kolmogorov widths and errors of optimal recovery -- for a given function class and establish relations between them. Recently, it was understood that in the case of recovery in
the $L_2$ norm the weighted least squares algorithms are reasonably good recovering methods.
Later, it was discovered that greedy type algorithms are also very good recovering methods.
We discuss these results in our paper.  The main point of this discussion is to explain that for obtaining those new results on sampling recovery researchers combined two deep and powerful techniques -- Lebesgue-type inequalities for the greedy type algorithms and theory of the universal sampling discretization.

 We begin with  very brief comments on classical problem of interpolation and recovery. We discuss the univariate periodic case. The main point is the step from the case of continuous functions and the uniform norm (see, for instance, \cite{Z}, volume 2, Ch. X) to the case of the integral norms $L_p$, $1\le p<\infty$,
(see, for instance, \cite{VT23}, \cite{VTbookMA}, Ch. 2, and \cite{DTU}).

We use $C$, $C'$ and $c$, $c'$ to denote various positive constants. Their arguments indicate the parameters, which they may depend on. Normally, these constants do not depend on a function $f$ and running parameters $m$, $v$, $u$. We use the following symbols for brevity. For two nonnegative sequences $a=\{a_n\}_{n=1}^\infty$ and $b=\{b_n\}_{n=1}^\infty$ the relation $a_n\ll b_n$ means that there is  a number $C(a,b)$ such that for all $n$ we have $a_n\le C(a,b)b_n$. Relation $a_n\gg b_n$ means that
 $b_n\ll a_n$ and $a_n\asymp b_n$ means that $a_n\ll b_n$ and $a_n \gg b_n$.
 For a real number $x$,  denote by $\lfloor x \rfloor$ the integer part of $x$, and by  $\lceil x\rceil$  the smallest integer that is greater than or equal to $x$.

Functions of the form
$$
t(x) = \sum_{|k|\le n}c_k e^{ikx} =a_0/2+\sum_{k=1}^n
(a_k\cos kx+b_k\sin kx)
$$
are called trigonometric polynomials of order at most  $n$. The set of such polynomials is denoted
by $\cT(n)$.

The  Dirichlet kernel of order $n$
$$
\mathcal D_n (x):= \sum_{|k|\le n}e^{ikx} = e^{-inx} (e^{i(2n+1)x} - 1)
(e^{ix} - 1)^{-1} =\frac{\sin (n + 1/2)x}{\sin (x/2)}
$$
is an even trigonometric polynomial.
Denote
$$
x^j := 2\pi j/(2n+1), \qquad j = 0, 1, ..., 2n.
$$
Clearly,  the points $x^j$,  $j = 1, \dots, 2n$,  are zeros of the Dirichlet
kernel $\mathcal D_n$ on $[0, 2\pi]$.
Therefore,  for any continuous function $f:[0, 2\pi)\to\mathbb{C}$,
 $$
I_n(f)(x) := (2n+1)^{-1}\sum_{j=0}^{2n} f(x^j) \mathcal D_n (x - x^j)
$$
interpolates $f$ at points $x^j$; that is,  $I_n(f)(x^j)= f(x^j)$, $j=0, 1, ..., 2n$.

It is easy to check that for any $t\in \cT(n)$ we have $I_n(t)=t$. Using this and the inequality
$$
\bigl|\mathcal D_n (x)\bigr|\le \min \bigl(2n+1, \pi/|x|\bigr), \qquad
 |x|\le \pi,
$$
we obtain the following Lebesgue inequality
$$
\|f-I_n(f)\|_\infty \le C(\ln(n+1))E_n(f)_\infty,
$$
where $E_n(f)_p$ is the best approximation of $f$ in the $L_p$ norm by polynomials from $\cT(n)$.

The de la Vall\'ee Poussin kernels are defined as follows
$$
\mathcal V_{n} (x) := n^{-1}\sum_{k=n}^{2n-1} \mathcal D_k (x) = \frac{\cos nx - \cos 2nx}{n (\sin (x/2))^2}.
$$

The  de la Vall\'ee Poussin kernels $\mathcal V_{n}$ are even trigonometric
polynomials of order $2n - 1$ with the majorant
$$
\left| \mathcal V_{n} (x) \right| \le C \min \left(n, \
 1/(nx^2)\right), \ |x|\le \pi.
$$

Consider the following  recovery operator (see \cite{VT23} and \cite{VT51})
\begin{equation}\label{rco}
R_n (f) := (4n)^{-1}\sum_{j=1}^{4n} f\left(x(j)\right)\mathcal V_n
\left(x - x(j)\right), \qquad x(j) := \pi j/(2n).
\end{equation}
It is easy to check that for any $t\in \cT(n)$ we have $R_n(t)=t$. Using this and the above majorant,
we obtain the following Lebesgue-type inequality (the de la Vall\'ee Poussin inequality)
$$
\|f-R_n(f)\|_\infty \le C E_n(f)_\infty.
$$

We now turn to the error in the $L_p$, $p\in [1,\infty)$. Let $\e:= \{\ep_k\}_{k=0}^\infty$ be a non-increasing sequence  of nonnegative numbers. Define
$$
E(\e,p) :=\{f\in \cC(\bbT):\, E_k(f)_p \le \ep_k,\, k=0,1,\dots\}.
$$
Here and throughout, $\bbT$ denotes the unit circle in $\mathbb{R}^2$, and as usual,  every function on $\bbT$ is identified as a $2\pi$-periodic function  on $\mathbb{R}$.

\begin{Theorem}[{\cite[Theorem 2]{VT23}}] Assume that a  non-increasing sequence  $\e$ of nonnegative numbers satisfies the following condition for some constants $B, D>0$: for all $s=0,1,\dots$,
$$
\sum_{\nu=s+1}^\infty \ep_{2^\nu} \le B\ep_{2^s},\qquad \ep_s\le D\ep_{2s}.
$$
Then for any $p\in [1,2]$, we have
$$
\sup_{f\in E(\e,p)} \|f-R_n(f)\|_p \asymp \sum_{\nu=0}^\infty 2^{\nu/p} \ep_{n2^\nu}.
$$
\end{Theorem}

Operators $R_n$ are not defined on $L_p$, when $p<\infty$.   Historically, the first idea was to
consider the operator $R_n J_r$ where
$$
J_r(f)(x) := (2\pi)^{-1}\int_\bbT f(x-y)F_r(y)dy,
$$
$$
 F_r(y):= 1+2\sum_{k=1}^\infty k^{-r}\cos(ky-r\pi/2).
$$
It was proved in \cite{VT23} that for $p\in [1,\infty]$ and $r>1/p$ we have ($I$ is the identity operator)
$$
\|I-R_nJ_r\|_{L_p\to L_p} \le C(r,p)n^{-r}.
$$
  The following inequalities turn out to be more convenient. Define
$$
V_s(f)(x) := (2\pi)^{-1} \int_\bbT f(x-y)\cV_s(y)dy,\  \ x\in\mathbb{R}.
$$
Then (\cite{VT51}) we have for $s\ge n$
$$
\|R_nV_s\|_{L_p\to L_p} \le C(s/n)^{1/p},\quad 1\le p\le \infty
$$
and
$$
\|I_nV_s\|_{L_p\to L_p} \le C(p)(s/n)^{1/p},\quad 1< p< \infty.
$$

We now proceed to the general setting.
Let $\Omega$ be a compact subset of $\bbR^d$ equipped with a Borel probability measure $\mu$. By $L_p$ norm of a real or complex valued function defined on $\Omega$ for $1\le p< \infty$,  we understand
$$
\|f\|_p:=\|f\|_{L_p(\Omega,\mu)} := \left(\int_\Omega |f|^pd\mu\right)^{1/p}.
$$
By $L_\infty$ norm we understand the uniform norm of continuous functions:
$$
\|f\|_\infty := \max_{\bx\in\Omega} |f(\bx)|,
$$
and with a little abuse of notations we sometimes write $L_\infty(\Omega)$ for the space $\cC(\Omega)$ of continuous functions on $\Omega$.

Recall the setting for the theory of optimal linear recovery. For a fixed $m\in\mathbb{N}$ and a set of points  $\xi:=\{\xi^j\}_{j=1}^m\subset \Omega$, let $\Phi $ be a linear operator from $\bbC^m$ into $L_p(\Omega,\mu)$.
For a class $\bF$ of functions on $\Omega$ (which is typically a  centrally symmetric and compact subset of $L_p(\Omega,\mu)$), define
$$
\varrho_m(\bF,L_p) := \inf_{\xi} \inf_{\text{linear}\, \Phi } \sup_{f\in \bF} \|f-\Phi(f(\xi^1),\dots,f(\xi^m))\|_p.
$$
The above described recovery procedure is a linear procedure.
The following modification of the above recovery procedure is also of interest. We now allow any mapping $\Phi : \bbC^m \to X_m \subset L_p(\Omega,\mu)$ where $X_m$ is a linear subspace of dimension $m$ and define
$$
\varrho_m^*(\bF,L_p) :=\inf_{\xi}  \inf_{\Phi;  X_m} \sup_{f\in \bF}\|f-\Phi(f(\xi^1),\dots,f(\xi^m))\|_p.
$$

In both of the above cases we build an approximant, which comes from a linear subspace of dimension at most $m$.
It is natural to compare quantities $\varrho_m(\bF,L_p)$ and $\varrho_m^*(\bF,L_p)$ with the
Kolmogorov widths. Let $\bF\subset L_p$ be a centrally symmetric compact. The quantities
$$
d_n (\bF, L_p) := \operatornamewithlimits{inf}_{\{u_i\}_{i=1}^n\subset L_p}
\sup_{f\in \bF}
\operatornamewithlimits{inf}_{c_i} \left \| f - \sum_{i=1}^{n}
c_i u_i \right\|_p, \quad n = 1, 2, \dots,
$$
are called the {\it Kolmogorov widths} of $\bF$ in the $L_p$. In the definition of
the Kolmogorov widths,  we take,  for each  $f\in \bF$, the element of best
approximation from $U := \sp \{u_i \}_{i=1}^n$  as an approximating element. This means
that in general (i.e. if $p\neq 2$) this method of approximation is not linear.

We have the following obvious inequalities:
\be\label{I1}
d_m (\bF, L_p)\le \varrho_m^*(\bF,L_p)\le \varrho_m(\bF,L_p).
\ee

Further, for a function class $\bF\subset \cC(\Omega)$,  we  define
$$
\varrho_m^o(\bF,L_p) := \inf_{\xi } \inf_{\cM} \sup_{f\in \bF}\|f-\cM(f(\xi^1),\dots,f(\xi^m))\|_p,
$$
where $\cM$ ranges over all  mappings $\cM : \bbC^m \to   L_p(\Omega,\mu)$  and
$\xi$ ranges over all subsets $\{\xi^1,\cdots,\xi^m\}$ of $m$ points in $\Omega$.
Here, we use the index {\it o} to mean optimality. Clearly, the above characteristic is a characteristic of nonlinear recovery.

The characteristics $\varrho_m$, $\varrho_m^*$, $\varrho_m^o$ and their variants are well studied for many particular classes of functions. For an exposition of known results we refer to the books \cite{TWW}, \cite{NoLN}, \cite{DTU}, \cite{VTbookMA}, \cite{NW1}--\cite{NW3} and references therein. The characteristics $\varrho_m^*$ and $\varrho_m$ are inspired by the concepts of the Kolmogorov width and the linear width. The quantity $\varrho_m^*$ appears to have been introduced in \cite{Dinh90}, $\varrho_m$ in \cite{VT51}, and $\varrho_m^o$   in \cite{TWW}.

In this paper we focus on the study of the two characteristics $\varrho_m$ and $\varrho_m^o$.
We show how the bounds for $\varrho_m(\bF,L_p)$ and $\varrho_m^o(\bF,L_p)$ can be derived from more delicate results on the Lebesgue-type inequalities. We now explain the setting of the corresponding problem.

{\bf Problem.} How to design a practical algorithm that gives  a  sampling recovery approximant  with an error comparable to the best possible approximation?

We discuss the setting of the above problem in the coming Subsection \ref{RILi}

\subsection{Lebesgue-type inequalities}
\label{RILi}

We need to introduce some  definitions  from the theory of the Lebesgue-type inequalities for greedy algorithms  (see  \cite[Section 8.7]{VTbookMA}).
Let $X$ be a Banach space with the norm $\|\cdot\|_X$,  and let $Y\subset X$ be a subspace of $X$ equipped with a stronger norm $\|\cdot\|_Y$ satisfying  $\|f\|_Y\ge \|f\|_X$ for all $f\in Y$.
 In a general setting,  we consider an algorithm (i.e., an approximation method) $\cA: = \{A_v(\cdot)\}_{v=1}^\infty$, which is  a  sequence
of mappings (linear or nonlinear) $A_v: Y\to A_v(Y) \subset Y$, $v=1,2,\cdots$, where $A_v(Y)$ is the range of the mapping $A_v$.  Clearly, for any $f\in Y$, we have
$$
\|f-A_v(f)\|_Y \ge  d(f,A_v(Y))_Y :=\inf_{g\in A_v(Y)}\|f-g\|_Y,\quad  v=1,2,\cdots .
$$
We are interested in those algorithms $\cA$, which provide approximation in the $X$-norm that is close to the best possible approximation in the stronger $Y$-norm.
To be more precise, we give the following  two  definitions (similar definitions can be found in \cite[Section 8.7]{VTbookMA} for  the case of $X=Y$).

\begin{Definition}\label{GD1} For a given integer $u\in\bbN$, we say that an algorithm  $\cA = \{A_v(\cdot)\}_{v=1}^\infty$ satisfies a Lebesgue-type inequality or  de la Vall{\'e}e Poussin inequality  of depth $u$  for the pair $(X,Y)$ of Banach spaces  if there exist constants $C_1\ge 1$ and   $C_2>0$ such that for any $f\in Y$,
\begin{equation}\label{AG}
\|f-A_{\lceil C_1v\rceil}(f)\|_X \le C_2 d(f,A_v(Y))_Y,\quad v=1,\dots,u.
\end{equation}
In the case $C_1=1$ we call them the Lebesgue inequalities.
\end{Definition}

More generally, we have

\begin{Definition}\label{GD2}  Let   $\mathbf{a}=\{a(j)\} _{j=1}^\infty$ be a given  sequence of positive numbers in $[1,\infty)$.  For a given integer  $u\in\bbN$,
	we say that an algorithm  $\cA = \{A_v(\cdot)\}_{v=1}^\infty$  satisfies
the Lebesgue $\mathbf{a}$-type inequalities   of depth $u$ for the pair $(X,Y)$ if there exists a constant $C_3>0$ such that for all $f\in Y$,
\begin{equation}\label{aG2}
\|f-A_{\lceil a(v)v\rceil}(f)\|_X \le C_3d(f,A_v(Y))_Y,\quad v=1,\dots,u.
\end{equation}
\end{Definition}

We discuss different algorithms $\cA$ here. In some cases (see \cite{VT183}), $\{A_v(Y)\}_{v=1}^\infty$ is
a sequence of $v$-dimensional subspaces, while in other cases (see, for instance, \cite{DTM3}),   $A_v(Y)$ is the set of all
$v$-term linear combinations of a given system $\cD$. We use two different types of sampling algorithms. The first type of algorithms is based on the $\ell_p$ minimization, whereas
the second type  is based on greedy algorithms. We now describe them in detail.

Let $X_N$ be an $N$-dimensional subspace of the space $\cC(\Omega)$ of continuous functions on $\Omega$. For a fixed $m\in\bbN$ and a set of points  $\xi:=\{\xi^\nu\}_{\nu=1}^m\subset \Omega$,  we associate with each  function $f\in \cC(\Omega)$ a vector (sample vector)
$$
S(f,\xi) := (f(\xi^1),\dots,f(\xi^m)) \in \bbC^m.
$$
Denote
$$
\|S(f,\xi)\|_p:=\|S(f,\xi)\|_{L_p^m}:= \left(\frac{1}{m}\sum_{\nu=1}^m |f(\xi^\nu)|^p\right)^{1/p},\quad 1\le p<\infty,
$$
and
$$
\|S(f,\xi)\|_\infty := \max_{1\leq \nu\leq m}|f(\xi^\nu)|.
$$
For a positive weight $\bw:=(w_1,\dots,w_m)\in \bbR_+^m$,  consider the following norm
$$
\|S(f,\xi)\|_{p,\bw}:= \left(\sum_{\nu=1}^m w_\nu |f(\xi^\nu)|^p\right)^{1/p},\quad 1\le p<\infty.
$$

Now we consider the following well known recovery operator (algorithm).\vspace{2mm}

{\bf Algorithm 1 ($\ell p\bw(\xi,X_N))$.}
$$
\ell p\bw(\xi)(f) := \ell p\bw(\xi,X_N)(f):=\underset {g\in X_N}{ \operatorname{argmin}}\,  \|S(f-g,\xi)\|_{p,\bw}.
$$
In the case $\bw_m:= (1/m,\dots,1/m)$,  we drop $\bw_m$ from the notation, and write
$$
\ell p(\xi,X_N)(f):=\ell p\bw_m(\xi,X_N)(f).
$$

For a given system $\cD_N= \{g_j\}_{j=1}^N$ of functions on $\Omega$,  and a positive  integer $v$, we denote by $\mathcal{X}_v(\cD_N)$ the collection of all linear subspaces spanned by   $v$ elements from $\cD_N$.

\vspace{2mm}
{\bf Algorithm 2 ($ \ell p^s(\xi,\cX_v(\cD_N))$, $B_v(\cdot,\cD_N,L_p(\xi))$.} For a given system $\cD_N$ and a set of points  $\xi:=\{\xi^\nu\}_{\nu=1}^m\subset \Omega$, define the algorithm
$$
L^s(\xi,f) :=   \underset {L\in \cX_v(\cD_N)} {\operatorname{argmin}} \|S(f-\ell p(\xi,L)(f),\xi)\|_p,
$$
\be\label{Alg2}
  \ell p^s(\xi,\cX_v(\cD_N))(f):= \ell p(\xi,L^s(\xi,f))(f).
\ee
Index $s$ here  stands for {\it sample} to stress that this algorithm only uses the sample vector $S(f,\xi)$.
Clearly, $\ell p^s(\xi,\cX_v(\cD_N))(f)$ is the best $v$-term approximation of $f$ with respect to $\cD_N$ in
the space $L_p(\xi) := L_p(\xi,\mu_m)$, where $\mu_m(\{\xi^\nu\}) =1/m$ for  $\nu=1,\dots,m$.
To stress this fact,  we use the notation $B_v(f,\cD_N,L_p(\xi)) := \ell p^s(\xi,\cX_v(\cD_N))(f)$.
Both notations $\ell p^s(\xi,\cX_v(\cD_N))$ and $B_v(\cdot,\cD_N,L_p(\xi))$ are used in the literature.

\vspace{2mm}
{\bf Algorithm 3.} This algorithm is a well known greedy algorithm -- the Weak Tchebychev Greedy Algorithm (the definition is given in Section \ref{gr}), which coincides with the Weak Orthogonal Matching Pursuit (Weak Orthogonal Greedy Algorithm) (see Section \ref{womp}) in the case of the Hilbert space $L_2(\Omega,\mu)$. We apply this algorithm
in the space $L_p(\xi,\mu_m)$, which means that we only use $S(f,\xi)$ and the restriction of the system $\cD_N$ on the set  $\xi$.

\vspace{2mm}

Algorithm 3 is good from the point of view of practical realization.  At each iteration this greedy algorithm searches over at most $N$ dictionary elements for choosing a new one and performs the $\ell_p$ projections on the appropriate subspace (alike the other two algorithms). On the other hand, the Algorithm 2 performs  $\binom{N}{v}$
iterations of the $\ell_p$ projections on the $v$-dimensional subspaces.

\subsection{Sampling recovery on function classes}
\label{RIfc}

In this subsection we emphasize importance of classes, which are defined by structural conditions rather than by smoothness conditions. We give a brief history of the development of this idea. The first result, which connected the best approximations $\sigma_v(f,\Psi)_2$ (see the definition in Section \ref{def}) with convergence of the series $\sum_{k}|\<f,\psi_k\>|$, was obtained by
S.B. Stechkin \cite{Stech} in 1955 in the case of an orthonormal system $\Psi$. We formulate his result and a more general result momentarily. The following classes were introduced in \cite{DeTe} in the  study of sparse approximation with respect to arbitrary systems (dictionaries) $\cD$ in a Hilbert space $H$. For a general dictionary $\cD\subset H$, and a parameter $\bt>0$,
define the class of functions (elements)
$$
\cA^o_\bt(\cD,M) := \left\{\sum_{k\in \Lambda} c_k g_k: \  \  g_k\in \cD,\ \ k\in\Lambda,  \  |\Lambda| <\infty,\
\sum_{k\in \Lambda} |c_k|^\bt \le M^\bt\right\},
$$
where $M>0$ is a constant, and  define $\cA_\bt(\cD,M)$  to be  the closure  of $\cA^o_\bt(\cD,M)$ in the Hilbert space $H$. Furthermore, we define $\cA_\bt(\cD)$ as the union of the classes $\cA_\bt(\cD,M)$ over all $M>0$.
In the case when  $\cD$ is an orthonormal system,  S.B. Stechkin \cite{Stech} proved (see a discussion of this result in \cite{DTU}, Section 7.4)
\be\label{Di1}
f\in \cA_1(\cD) \quad \text{if and only if}\quad \sum_{v=1}^\infty (v^{1/2}\sigma_v(f,\cD)_{H})\frac{1}{v} <\infty.
\ee
A version of (\ref{Di1}) for the classes $\cA_\bt(\cD)$, $\bt\in (0,2)$, was obtained in \cite{DeTe}:
\be\label{Di2}
f\in \cA_\bt(\cD) \quad \text{if and only if}\quad \sum_{v=1}^\infty (v^{\alpha}\sigma_v(f,\cD)_H)^\bt\frac{1}{v} <\infty,
\ee
where $\al := 1/\bt-1/2$. In particular, (\ref{Di2}) implies that for $f\in \cA_\bt(\cD)$, $\bt\in (0,2)$, we have
\be\label{Di3}
\sigma_v(f,\cD)_H \ll v^{1/2-1/\bt}.
\ee

 We recall that relations (\ref{Di1}) -- (\ref{Di3}) were proved for an orthonormal system $\cD$.
 It is very interesting to note that (\ref{Di3}) actually holds for a general normalized system $\cD$ provided
 $\bt \in (0,1]$. In the case $\bt =1$ it was proved by B. Maurey (see \cite{Pi}). For $\bt \in (0,1]$ it was proved in \cite{DeTe}. We note that classes $\cA_1(\cD)$ and their generalizations for the case of Banach spaces play a fundamental role in studying best $v$-term approximation and convergence properties of greedy algorithms with respect to general dictionaries $\cD$ (see \cite{VTbook}). This fact encouraged experts to introduce and study function classes defined by restrictions on the coefficients of the functions' expansions. We explain this approach in detail.

 Let    $\Psi:=\{\psi_\bk\}_{\bk\in \bbZ^d}$ be a system indexed by $\bk\in\bbZ^d$. Consider a sequence of subsets $\cG:=\{G_j\}_{j=1}^\infty$, $G_j \subset \bbZ^d$, $j=1,2,\dots$, such that
 \be\label{Di4}
 G_1\subset G_2\subset \cdots \subset G_j\subset G_{j+1} \subset \cdots,\qquad \bigcup_{j=1}^\infty G_j =\bbZ^d.
 \ee
Consider functions representable in the form of absolutely convergent series:
\be\label{Direpr}
f = \sum_{\bk\in\bbZ^d} a_\bk(f)\psi_\bk,\qquad \sum_{\bk\in\bbZ^d} |a_\bk(f)|<\infty.
\ee
For $\bt \in (0,1]$ and $r>0$ consider the  class $\bA^r_\bt(\Psi,\cG)$ of all functions $f$ which have representations (\ref{Direpr}) satisfying  the  conditions
\be\label{DiAr}
  \left(\sum_{  \bk\in G_j\setminus G_{j-1}} |a_\bk(f)|^\bt\right)^{1/\bt} \le 2^{-rj},\quad j=1,2,\dots,\quad G_0 :=\emptyset  .
\ee
One can also consider the following narrower version $\bA^r_\bt(\Psi,\cG,\infty)$,  the class of all functions $f$ which have representations (\ref{Direpr}) satisfying the  condition
\be\label{Di7}
\sum_{j=1}^\infty 2^{r\bt j} \sum_{  \bk\in G_j\setminus G_{j-1}} |a_\bk(f)|^\bt \le 1, \quad G_0 :=\emptyset  .
\ee

Probably, the first realization of the idea of the classes $\bA^r_\bt(\Psi,\cG)$ was done in \cite{VT150} in the special case, when $\Psi$ is the trigonometric system $\cT^d := \{e^{i(\bk,\bx)}\}_{\bk\in\bbZ^d}$. We now proceed to the definition of classes $\bW^{a,b}_A(\cT^d)$ from \cite{VT150}, which corresponds to the case $\bt=1$. Introduce the necessary notations.
Let $\mathbf s=(s_1,\dots,s_d )\in\mathbb{N}_0^d$ be a  vector  whose  coordinates  are
nonnegative integers, where $\mathbb{N}_0=\mathbb{N}\cup \{0\}$. Let
$$
\rho(\mathbf s) := \bigl\{ \mathbf k\in\mathbb Z^d:\  \lfloor 2^{s_j-1}\rfloor \le
|k_j| < 2^{s_j},\qquad j=1,\dots,d \bigr\}.
$$
For $f\in L_1 (\bbT^d)$, define
$$
\delta_{\mathbf s} (f,\mathbf x) :=\sum_{\mathbf k\in\rho(\mathbf s)}
\hat f(\mathbf k)e^{i(\mathbf k,\mathbf x)},\quad \hat f(\mathbf k) := (2\pi)^{-d}\int_{\bbT^d} f(\bx)e^{-i(\mathbf k,\mathbf x)}d\bx.
$$
Consider functions with absolutely convergent Fourier series. For such functions,  define the Wiener norm (the $A$-norm or the $\ell_1$-norm)
$$
\|f\|_A := \sum_{\mathbf k\in\mathbb{Z}^d}|\hat f({\mathbf k})|.
$$
The following classes, which are convenient in studying sparse approximation with respect to the trigonometric system,
were introduced and studied in \cite{VT150}. Define, for $f\in L_1(\bbT^d)$,
$$
f_j:=\sum_{\|\bs\|_1=j}\delta_\bs(f), \quad j\in \bbN_0.
$$
For parameters $ a\in \bbR_+$ and  $ b\in \bbR$,  define the class
$$
\bW^{a,b}_A:=\Big\{f\in L^1(\bbT^d):\  \  \|f_j\|_A \le 2^{-aj}(\bar j)^{(d-1)b},\quad \forall  j\in \bbN_0\Big\},
$$
where $\bar j:=\max(j,1)$.
In this case, define
$$
G_j := \bigcup_{\bs:\  \|\bs\|_1 \le j} \rho(\bs), \quad j=1,2,\dots,
$$
and the classes $\bW^{a,b}_A$ with $b=0$ are similar to the above defined classes $\bA^a_1(\cT^d,\cG)$. A more narrow version $\bA^a_1(\cT^d,\cG,\infty)$ of these classes
was studied recently in \cite{JUV}.

The classes $\bA^r_\bt(\Psi)$ studied in \cite{VT203} correspond to the classes $\bA^r_\bt(\Psi,\cG)$ with
\be\label{Di8}
G_j:= \{\bk\in \bbZ^d\,:\, \|\bk\|_\infty <2^j\},\quad j=1,2,\dots.
\ee
Note that the classes $\bA^r_\bt(\cT^d)$ are related to the periodic isotropic Nikol'skii classes $H^r_q$. There are several equivalent definitions of the classes $H^r_q$. We give a  definition that is most convenient to us. Let $r>0$ and $1\le q\le \infty$. The class $H^r_q$ consists of all periodic functions $f$ in $d$ variables satisfying the conditions
$$
\left\|\sum_{\lfloor2^{j-1}\rfloor \le \|\bk\|_\infty< 2^j} \hat f(\bk) e^{i(\bk,\bx)}\right\|_q \le 2^{-rj}, \quad j=0,1,\dots.
$$
For instance, it is easy to see that in the case when  $q=2$ and $r >d/2$,  the class $H^r_q$ is embedded into the class $\bA^{r-d/2}_1(\cT^d)$. However, the class $\bA^{r-d/2}_1(\cT^d)$ is
substantially larger than $H^r_2$. For instance, take $\bk \in G_j\setminus G_{j-1}$ with $G_j$ defined in (\ref{Di8}). Then the function $g_\bk := 2^{-(r-d/2)j}e^{i(\bk,\cdot)}$ belongs to $\bA^{r-d/2}_1(\cT^d)$ but does not belong to any $H^{r'}_2$ with $r'> r-d/2$.

We now give a brief comparison of sampling recovery results for classes $H^r_q$ and $\bA^r_\bt(\cT^d)$.
It is known (see \cite{VTbookMA}, Theorem 3.6.4, p.125) that for all $1\le p,q\le \infty$, $r >d/q$, we have
\be\label{Di9}
\varrho_m(H^r_q,L_p) \asymp m^{-r/d +(1/q-1/p)_+},
\ee
where $ (a)_+ := \max(a,0)$ for $a\in\mathbb{R}$.
Clearly, (\ref{Di9}) implies the same upper bound for the $\varrho^o_m(H^r_q,L_p)$.
Results on the lower bounds for $\varrho^o_m(\cT(\bN,d)_q,L_p)$ (see the definition of  $\cT(\bN,d)_q$ in Subsection \ref{Lb}) from \cite{VT202} (see Lemma 4.3 there), Lemma 4.1 of \cite{VT203}, and Lemma \ref{NLL1} below show that the following relation holds
\be\label{Di10}
\varrho^o_m(H^r_q,L_p) \asymp m^{-r/d +(1/q-1/p)_+}.
\ee
In particular,
\be\label{Di11}
\varrho^o_m(H^r_2,L_p) \asymp m^{-r/d +(1/2-1/p)_+}.
\ee
The lower bound (1.5) from \cite{VT202} and the lower bound in (1.13) of \cite{VT203} give the following bound in the case $\bt=1$:
\be\label{Di12}
\varrho^o_m(\bA^r_1(\cT^d,L_p) \gg m^{-1/2 +(1/2-1/p)_+ -r/d}.
\ee
The upper bound in (1.13) of \cite{VT203} gives the following bound
\be\label{Di13}
\varrho^o_m(\bA^r_1(\cT^d,L_p) \ll \left(\frac{m}{(\log m)^3}\right)^{-1/2 +(1/2-1/p)_+ -r/d}.
\ee
Relations (\ref{Di11})--(\ref{Di13}) mean close bounds for the class $H^r_2$
and for the larger class $\bA^{r-d/2}_1(\cT^d)$, $r>d/2$. Note that for the class $H^r_2$ we
obtain the same bounds for the linear sampling recovery. It is proved in \cite{VT202} that for
the linear sampling recovery in $\bA^r_\bt(\cT^d)$ we have
\be\label{Di14}
\varrho_m(\bA^r_\bt(\cT^d,L_2) \gg m^{-r/d}.
\ee
Inequalities (\ref{Di14}) with $\bt=1$ and (\ref{Di13}) with $p=2$ demonstrate that nonlinear
sampling recovery provides better error guarantees than linear sampling recovery.

{\bf A brief comment on important steps.} We now list some steps, which played an important role in the very recent development of the theory of nonlinear sampling recovery theory.

{\bf Step 1.} The authors of \cite{JUV} discovered that for some classes $\bF$ of periodic functions,
the best $v$-term approximation in the uniform norm of these classes $\sigma_v(\bF,\cT^d)_{L_\infty}$ with respect to the trigonometric system
can be used for estimating from above the $\varrho_m^o(\bF,L_2)$.

{\bf Step 2.} The authors of the papers \cite{DTM1}--\cite{DTM3} started to use the universal
discretization of integral norms in the sampling recovery.

{\bf Step 3.} The greedy type and the $\ell_p$ minimization type algorithms, which provide
sparse approximants, were used in the papers \cite{DTM2} and \cite{DTM3}. The Lebesgue-type inequalities for these algorithms were proved.

{\bf Step 4.} It was observed in \cite{DTM1} and \cite{DTM3} that one can use the best $v$-term approximation in
the norm $L_p(\Omega,\mu_\xi)$ with $p<\infty$ (see the definition of the $\mu_\xi$ in (\ref{muxi}) below), which is weaker than the uniform norm, instead of the uniform norm. This allowed
us to use known deep results on sparse approximation in Banach spaces  to improve
the known upper bounds for $\varrho_m^o(\bF,L_p)$.

\section{Some definitions and notations}
\label{def}

\subsection{Universality and incoherence}

  We begin with a brief
description of  some  necessary concepts on
sparse approximation.   Let $X$ be a Banach space with norm $\|\cdot\|:=\|\cdot\|_X$, and let $\cD=\{g_i\}_{i=1}^\infty $ be a given (countable)  system of elements in $X$. Given a finite subset $J\subset \bbN$, we define $V_J(\cD):=\sp\{g_j:\  \ j\in J\}$.
For a positive  integer $ v$, we denote by $\mathcal{X}_v(\cD)$ the collection of all linear spaces $V_J(\cD)$  with   $|J|=v$, and
denote by $\Sigma_v(\cD)$  the set of all $v$-term approximants with respect to $\cD$; that is,
$
\Sigma_v(\cD):= \bigcup_{V\in\cX_v(\cD)} V.
$
Given $f\in X$,  we define
$$
\sigma_v(f,\cD)_X := \inf_{g\in\Sigma_v(\cD)}\|f-g\|_X=d(f,\Sigma_v(\cD))_X,\  \ v=1,2,\cdots.
$$
Moreover,   for a function class $\bF\subset X$, we define $\sigma_0(\bF,\cD)_X := \sup_{f\in\bF} \|f\|_X
 $, and
$$
 \sigma_v(\bF,\cD)_X := \sup_{f\in\bF} \sigma_v(f,\cD)_X,\quad  v=1,2,\cdots.$$

 In this paper we mostly consider the case where  $X=L_p(\Omega,\mu)$ and $1\le p<\infty$.  In this case,  we sometimes write for brevity $\sigma_v(\cdot,\cdot)_p$ instead of
$\sigma_v(\cdot,\cdot)_{L_p(\Omega,\mu)}$.

We study systems, which have  the universal sampling discretization property.

  \begin{Definition}\label{ID1} Let $1\le p<\infty$. We say that a finite subset  $\xi:= \{\xi^j\}_{j=1}^m \subset \Omega $ provides the {\it  $L_p$-universal sampling discretization}   for the collection $\cX:= \{X(n)\}_{n=1}^k$ of finite-dimensional  linear subspaces $X(n)$ if
 \be\label{ud}
\frac{1}{2}\|f\|_p^p \le \frac{1}{m} \sum_{j=1}^m |f(\xi^j)|^p\le \frac{3}{2}\|f\|_p^p\quad \text{for any}\quad f\in \bigcup_{n=1}^k X(n) .
\ee
We denote by $m(\cX,p)$ the minimal integer  $m$ such that there exists a set $\xi$ of $m$ points which
provides  the $L_p$-universal sampling discretization (\ref{ud}) for the collection $\cX$.

We will use a brief form $L_p$-usd for the $L_p$-universal sampling discretization (\ref{ud}).
\end{Definition}

In \cite{DTM3} the following one-sided universal discretization condition on the collection was used.

 \begin{Definition}[\cite{DTM3}]\label{ID1a}
     Let $1\le p <\infty$. We say that a set $\xi:= \{\xi^j\}_{j=1}^m \subset
     \Omega $ provides the {\it one-sided $L_p$-universal sampling discretization with constant $D\ge 1$}   for a
     collection $\cX:= \{X(n)\}_{n=1}^k$ of finite-dimensional  linear subspaces
     $X(n)$  if we have
     \be\label{I3a}
     \|f\|_p \le D\left(\frac{1}{m} \sum_{j=1}^m |f(\xi^j)|^p\right)^{1/p} \quad \text{for any}\quad f\in \bigcup_{n=1}^k X(n) .
     \ee
 \end{Definition}
  Here is the definition of somewhat weaker condition than the above one.

\begin{Definition}[{\cite{LMT}}]\label{UDD1}
    Let $1\le p <\infty$. We say that a set $\xi:= \{\xi^j\}_{j=1}^m \subset
    \Omega $ provides  universal LDI$(p,\infty)$ (LDI stands for Left Discretization Inequality)  with constant  $D\ge 1$ for a collection $\cX:=
    \{X(n)\}_{n=1}^k$ of finite-dimensional linear subspaces $X(n)$ if
    \be\label{I3}
    \|f\|_p \le D\max_{1\le j\le m} |f(\xi^j)| \quad \text{for any}\quad f\in \bigcup_{n=1}^k X(n),
    \ee
    in which case we write $\bigcup\limits_{n=1}^k X(n)\in \LD(m,p,\infty,D)$.
\end{Definition}

The  property given in  Definition \ref{ID2} below concerns incoherence property of the system, which is
known to be useful in approximation by the WCGA (see, for instance, \cite{VTbookMA}, Section 8.7).

\begin{Definition}\label{ID2}  Let $X$ be a Banach space with a norm $\|\cdot\|$. Let $v, S\in\mathbb{N}$ be given integers such that $v\leq S$.  We say that a system $\cD=\{g_i\}_{i=1}^\infty\subset X$ has   ($v,S$)-incoherence property with parameters   $V>0$ and $r>0$ in the space $X$  if, whenever  $A\subset B\subset \mathbb{N}$,   $|A|\le v$,   $|B|\le S$ and $\{c_i\}_{i\in B}\subset \bbC$,  we have
\be\label{incoh}
\sum_{i\in A} |c_i| \le V|A|^r\left\|\sum_{i\in B} c_ig_i\right\|.
\ee
We will use a brief form ($v,S$)-ipw($V,r$)   for the ($v,S$)-incoherence property with parameters   $V>0$ and $r>0$ in $X$.
\end{Definition}

We gave Definition \ref{ID2} for a countable system $\cD$. Similar definition can be given for any system $\cD$ as well.

 Denote, for a set $\xi:=\{\xi^1,\cdots, \xi^m\}$ of $m$ points from $\Omega$,
\be\label{muxi}
\Omega_m:=\{\xi^1,\cdots, \xi^m\},\   \  \mu_m:= \frac1{m}\sum_{j=1}^m \delta_{\xi^j},\  \  \ \text{and}\  \ \mu_\xi := \frac{\mu+\mu_m}{2},
\ee
where $\delta_\bx$ denotes the Dirac measure supported at a point $\bx$.

 As in  previous  papers (see, for instance, \cite{VT202}),  we study a special case when $\cD$   satisfies certain restrictions. Following notations used in the literature, for convenience, in this case we use the notation $\Psi$ for a system $\cD$.
 We formulate these restrictions in the form of conditions imposed on $\Psi$.

{\bf Condition A.} Assume that $\Psi:=\{\ff_j\}_{j=1}^\infty$ is a system of uniformly bounded functions on $\Omega \subset \bbR^d$ satisfying
\be \label{ub}
\sup_{\bx\in\Omega} |\vi_j(\bx)|\leq 1,\   \ 1\leq j<\infty.
\ee

{\bf Condition B1.} Assume that  $\Psi:=\{\ff_j\}_{j=1}^\infty$ is an orthonormal system in $L_2(\Omega, \mu)$.

{\bf Condition B2.} Assume that $\Psi:=\{\ff_j\}_{j=1}^\infty$ is a Riesz  system in the space $L_2(\Omega, \mu)$;  i.e. there exist two positive constants $0< R_1 \le R_2 <\infty$ such that
for any $N\in\bbN$ and any $(a_1,\cdots, a_N) \in\bbC^N,$
\begin{equation}\label{Riesz}
R_1 \left( \sum_{j=1}^N |a_j|^2\right)^{1/2} \le \left\|\sum_{j=1}^N a_j\ff_j\right\|_2 \le R_2 \left( \sum_{j=1}^N |a_j|^2\right)^{1/2}.
\end{equation}

{\bf Condition B3.} Assume that  $\Psi$ is a Bessel system, i.e. there exists a constant $K>0$ such that for any $N\in\bbN$ and   for any $(a_1,\cdots, a_N) \in\bbC^N,$
\begin{equation}\label{Bessel}
  \sum_{j=1}^N |a_j|^2 \le K  \left\|\sum_{j=1}^N a_j\ff_j\right\|^2_2 .
\end{equation}

Clearly, Condition B1 implies Condition B2 with $R_1=R_2=1$. Condition B2 implies Condition B3 with
$K=R_1^{-2}$.

We often use the concept of Nikol'skii inequality. For the reader's convenience we formulate it here. Other definitions and notations are introduced below in the text.

{\bf Nikol'skii-type inequalities.} Let $1\le p\le q\le\infty$,  and let $X_N\subset L_q(\Omega)$ be a subspace of dimension $N$. The inequality
\begin{equation}\label{I4NI}
\|f\|_q \leq H\|f\|_p,\   \ \forall f\in X_N
\end{equation}
is called the Nikol'skii inequality for the pair $(p,q)$ with the constant $H$.
In such a case,  we write $X_N \in NI(p,q,H)$. Typically, $H$ depends on $N$, for instance, $H$ can be of order $N^{\frac{1}{p}-\frac{1}{q}}$.

 In addition to the above formulated Nikol'skii inequality we need one more property of a similar nature.

{\bf $u$-term Nikol'skii inequality.} Let $1\le p\le q\le \infty$ and let $\Psi$ be a system from $L_q:=L_q(\Omega,\mu)$. For a natural number $u$, we say that the system $\Psi$ has
the $u$-term Nikol'skii inequality for the pair $(p,q)$ with the constant $H$ if the following
inequality holds
\be\label{vNI}
\|f\|_q \le H\|f\|_p,\qquad \forall f\in \Sigma_u(\Psi).
\ee
 In such a case,  we write $\Psi \in NI(p,q,H,u)$.

 Note, that obviously $H\ge 1$.

\subsection{Sparse approximation in Banach spaces}

Recall that  the modulus of smoothness of a Banach space  $X$ is the function on $(0,\infty)$ defined as
\begin{equation}\label{CG1}
\eta(X,w):=\sup_{x,y\in X;
		\|x\|= \|y\|=1}\left(\frac{\|x+wy\|+\|x-wy\|}{2} -1\right),\  \ w>0,
\end{equation}
and that $X$ is called uniformly smooth  if  $\eta(w)/w\to 0$ when $w\to 0+$.
It is well known that the $L_p$ space with $1< p<\infty$ is a uniformly smooth Banach space with
\be\label{CG2}
\eta(L_p,w)\le \begin{cases}(p-1)w^2/2, & 2\le p <\infty,\\   w^p/p,& 1\le p\le 2.
\end{cases}
\ee

We will use some known general results on best $v$-term approximations with respect to
an arbitrary system in a Banach space. Usually, these results are proved in the case of real Banach spaces. It is convenient for us to consider complex Banach spaces, partially because of our applications to the special case when the system of interest is the trigonometric system $\cT^d$. We now prove the complex version of the result known in the case of real Banach spaces. For a system $\cD\subset X$,  denote
 $$
 A_1(\cD) := \left\{f\in X:\, f=\sum_{i=1}^\infty a_ig_i,\quad g_i\in \cD,\quad a_i\in\mathbb{C},\  \  \sum_{i=1}^\infty |a_i| \le 1 \right\}.
 $$

\begin{Lemma}[{\cite{VT205}}]\label{NLv} Let $\cD$ be a system of elements in a complex Banach space $X$ satisfying $\sup_{g\in \cD}\|g\|_X \le 1$. Assume that
the modulus of smoothness of  $X$ satisfies the following condition for  some constants $1<q\le 2$ and $\gamma>0$:
  $$\eta(X,w) \le \gamma w^q\  \  \text{ for all $w>0$}.$$  Then, there exists a constant $C(q,\gamma)>0$, which may only depend on $q$ and $\gamma$, such that
\be\label{Nv}
\sigma_v(A_1(\cD),\cD)_X \le C(q,\gamma) (v+1)^{1/q-1}\   \ \text{for any $v\in\mathbb{N}_0$}.
\ee
Moreover, the bound in (\ref{Nv}) is provided by a constructive method based on greedy algorithms.
\end{Lemma}
\begin{proof} In the case of real Banach spaces Lemma \ref{NLv} is known (see, for instance, \cite{VTbook}, p.342). The corresponding bound is provided by the Weak Tchebychev Greedy Algorithm. We derive the complex case from the real one.

By a standard decomplexification process, every complex Banach space $(X, \|\cdot\|_X)$ can be viewed as a direct sum of two identical real Banach spaces representing the ``real'' and ``imaginary'' parts of its vectors. To be precise, let $X_{\mathbb{R}}=(X_{\mathbb{R}},\|\cdot\|_X) $ denote the underlying real Banach space obtained by restricting the scalar multiplication of $X$ to the real  field $\mathbb{R}$. Then
 $X$ can be decomposed as the direct sum $X = Y \oplus iY$, where $Y$ is a  real Banach  subspace of $X_{\mathbb{R}}$. Furthermore,  by the open mapping theorem,  there exists a constant $c > 0$ such that for all $x, y \in Y$,
 $$c (\|x\|_X + \|y\|_X) \leq \|x + iy\|_X \leq \|x\|_X + \|y\|_X.$$

 In this context, every vector $f \in X$ admits a unique representation $$f = f^R + i f^I, $$where  $f^R = \operatorname{Re}(f)\in Y$,  $f^I = \operatorname{Im}(f)\in Y$, and $$\|f\|_X\sim \|f^R\|_X+\|f^I\|_X.$$

 Now consider the systems
$$
\cD^r_1 := \{g^R\, :\, g\in \cD\}\  \ \text{and}\  \  \cD^r_2 := \{g^I\, :\, g\in \cD\}.
$$
For $i=1, 2$, we define  $$X^r_i=\overline{\sp (\cD_i^r)}^{ X_{\mathbb{R}}}$$
 to be the real Banach subspace of $(X_{\mathbb{R}},\|\cdot\|_X)$ spanned by $\cD_i^r$.
 Clearly, for $i=1,2$,
 \[\eta(X^r_i,w) \le \eta(X,w)\   \ \text{ and }\  \  \sup_{g\in \cD_i^r}\|g\|_X \leq C \sup_{g\in \cD}\|g\|_X\leq C.\]
 Thus, by applying the real-space analogue of Lemma \ref{NLv} to the systems $\cD_i^r$ in the real Banach spaces $X_i^r$,  we obtain
\be\label{Nv1}
\sigma_v(A_1(\cD^r_i),\cD^r_i)_{X_{\mathbb{R}}} \le C_1(q,\gamma) (v+1)^{1/q-1},\quad i=1,2.
\ee
Next, let $f\in X$ be such that
$$
f=\sum_{k=1}^\infty c_kg_k,\quad g_k \in \cD,\quad \sum_{k=1}^\infty |c_k| \le 1.
$$
Writing $c_k = a_k +ib_k$ with $a_k, b_k\in\mathbb{R}$, we obtain
$$
f = \sum_{k=1}^\infty (a_kg^R_k -b_kg_k^I) +i \sum_{k=1}^\infty (a_kg^I_k +b_kg_k^R)=: f^R + i f^I.
$$
Clearly,
$$
|a_k|+|b_k| \le \sqrt{2}(a_k^2+b_k^2)^{1/2} =  \sqrt{2}|c_k|.
$$
Therefore,
$$
2^{-1/2}f^R\in \conv(A_1(\cD^r_1),A_1(\cD^r_2))\  \ \text{and}\  \  2^{-1/2}f^I\in \conv(A_1(\cD^r_1,A_1(\cD^r_2)).
$$
It follows that
$$
\sigma_{4v}(A_1(\cD),\cD)_X \le 2\max_{i=1,2}\sigma_v(A_1(\cD^r_i),\cD^r_i)_{X_R}.
$$
This and (\ref{Nv1}) imply Lemma \ref{NLv}.

\end{proof}

Note that the papers \cite{DGHKT} and \cite{VT209} are devoted to the study of greedy algorithms in complex Banach spaces. In particular, Lemma \ref{NLv} follows from Theorem 3.2 of \cite{VT209}.

\section{Some discretization results}
\label{dis}

\subsection{Sampling discretization in finite-dimensional subspaces}
\label{disd}

There are several survey papers on sampling discretization in finite-dimensional subspaces -- \cite{DPTT}, \cite{KKLT}, and \cite{LMT}. We provide here a detailed discussion of only those results which are used in this paper.
It is well known that results on sampling discretization in the $L_2$-norm imply some results on sampling discretization in the $L_\infty$-norm. We will illustrate this phenomenon on some known examples.
Probably, the first example of this type is the multivariate trigonometric polynomials.
By $Q$ we denote a finite subset of $\bbZ^d$, and $|Q|$ stands for the number of elements in $Q$. Let
$$
\cT(Q):= \left\{f: f=\sum_{\bk\in Q}c_\bk e^{i(\bk,\bx)},\  \  c_{\bk}\in\mathbb{C}\right\}.
$$
The following Theorem \ref{TrD} was proved in \cite{VT158}.

\begin{Theorem}[{\cite[Theorem~1.1]{VT158}}] \label{TrD}There are three positive absolute constants $C_1$, $C_2$, and $C_3$ with the following properties: For any $d\in \bbN$ and any $Q\subset \bbZ^d$   there exists a set of  $m \le C_1|Q| $ points $\xi^j\in \bbT^d$, $j=1,\dots,m$ such that for any $f\in \cT(Q)$,
    we have
    \begin{equation}\label{3-1}
    C_2\|f\|_2^2 \le \frac{1}{m}\sum_{j=1}^m |f(\xi^j)|^2 \le C_3\|f\|_2^2.
    \end{equation}
\end{Theorem}

In \cite{DPTT} it was shown how Theorem \ref{TrD} implies a result on sampling discretization of the $L_\infty$-norm. Namely, the following Theorem \ref{TrDi} was proved.

\begin{Theorem}[\cite{DPTT}] \label{TrDi} Let $C_1$, $C_2$, and $C_3$ be the  three positive absolute constants from Theorem \ref{TrD}. Then for any $d\in \bbN$ and any $Q\subset \bbZ^d$   there exists a set $\xi$ of  $m \le C_1|Q| $ points $\xi^j\in \bbT^d$, $j=1,\dots,m$, such that for any $f\in \cT(Q)$, \eqref{3-1} holds, and
$$
\|f\|_\infty \le C_2^{-1/2}|Q|^{1/2}\left(\frac{1}{m}\sum_{j=1}^m |f(\xi^j)|^2\right)^{1/2}\le  C_2^{-1/2}|Q|^{1/2}\max_{1\le j\le m}|f(\xi^j)|.
$$
\end{Theorem}
\begin{proof} For the reader's convenience we present the one line proof from \cite{DPTT} here. We use the set of points provided by Theorem \ref{TrD}. Then $m\le C_1|Q|$ and for any $f\in\cT(Q)$ we have
\begin{align*}
\|f\|_\infty & \le |Q|^{1/2}\|f\|_2 \le |Q|^{1/2} C_2^{-1/2} \left(\frac{1}{m}\sum_{j=1}^m |f(\xi^j)|^2\right)^{1/2} \\ & \le |Q|^{1/2} C_2^{-1/2} \max_{1\le j\le m}|f(\xi^j)|.
\end{align*}
\end{proof}

We point out that in the above proof in addition to the sampling discretization result --Theorem \ref{TrD} -- the Nikol'skii inequality $\|f\|_\infty \le  |Q|^{1/2}\|f\|_2$ has been used.
The following  sampling discretization result shows that the Nikol'skii inequality assumption guarantees good discretization inequalities for general subspaces.

\begin{Theorem}[{\cite{LimT}}]\label{LimTT} Let  $\Omega\subset \bbR^d$ be a   nonempty set equipped with  a probability measure $\mu$. Assume that $X_N \in NI(2,\infty, KN^{1/2})$ for some constant $K\ge 1$.
Then there exist absolute  constants $C_1, C_2, C_3>0$  and  a set $\{\xi^j\}_{j=1}^m\subset \Omega$ of $m \le C_1 K^2 N$ points such that for any $f\in X_N$,  we have
\begin{equation*}
C_2 \|f\|_2^2 \le \frac{1}{m}\sum_{j=1}^m |f(\xi^j)|^2 \le C_3 K^2\|f\|_2^2.
\end{equation*}
\end{Theorem}

One can prove discretization results without the Nikol'skii inequality assumption. However, known results in this direction provide discretization with weights.

\begin{Theorem}[{\cite{LimT}}]\label{LimTTw}
Given an arbitrary  complex  linear  subspace $X_N$ of $L_2(\Omega,\mu)$ of dimension $N$,  there exist     a set of $m\leq   C_1'N$ points $\xi^1,\ldots, \xi^m\in\Omega$, and a set of
    nonnegative  weights $w_1,\cdots, w_m\ge 0$  such that
    $$
    c_0'\|f\|_2^2\leq  \sum_{j=1}^m w_j |f(\xi^j)|^2 \leq  C_0'
    \|f\|_2^2,\quad   \forall f\in X_N,
    $$
    where $C_1'$, $c_0'$, $C_0'$ are absolute positive constants.
\end{Theorem}

Here is a more precise version of Theorem \ref{LimTTw} with the constant $C_1'$ being close to one.

\begin{Theorem}[{\cite{DPSTT2}}]\label{disdT1}If $X_N$ is a real $N$-dimensional subspace of $L_2(\Omega,\mu)$, then for any $b\in (1,2]$, there exist a set of $m\leq bN$ points $\xi^1,\ldots, \xi^m\in\Omega$ and a set of nonnegative  weights $w_j$, $j=1,\ldots, m$  such that
\be\label{disd1}
\|f\|_2\leq \left( \sum_{j=1}^m w_j |f(\xi^j)|^2 \right)^{1/2} \leq  \frac{C}{b-1}  \|f\|_2,\  \qquad \forall f\in X_N,
\ee
where $C>1$ is an absolute constant.
\end{Theorem}

Theorems \ref{LimTTw} and \ref{disdT1} were derived from deep results established  in
\cite{BSS}. We refer the reader to \cite[Section 6]{DPSTT2} for a detailed discussion of related results.

  In \cite{KKT} it was shown how Theorem \ref{LimTTw} implies a result on sampling discretization of the $L_\infty$-norm. Namely, the following Theorem~\ref{BP1} was proved.

\begin{Theorem}[\cite{KKT}]\label{BP1} There are two absolute constants $c_1$ and $c_2$ such that for any
$X_N \in NI(2,\infty, M)$  there exists a set of $m\leq c_1N$ points
$\xi^1,\ldots, \xi^m\in\Omega$ with the property: For any $f\in X_N$ we have
\be\label{A-AT2}
\|f\|_\infty \le c_2 M \max_{1\le j\le m} |f(\xi^j)|.
\ee
\end{Theorem}

Note that in the special case of  $M=KN^{1/2}$, the proof of Theorem \ref{BP1} from \cite{KKT} shows that Theorem \ref{LimTT} implies the following result on simultaneous  sampling discretization of the $L_\infty$ and $L_2$ norms.

\begin{Theorem}\label{BP1a} Let  $\Omega\subset \bbR^d$ be a   nonempty set equipped  with  a probability measure $\mu$. Assume that $X_N \in NI(2,\infty, KN^{1/2})$ for some constant $K\ge 1$. Let $C_1$, $C_2$, $C_3$ denote the three positive absolute  constants from Theorem \ref{LimTT}.
Then  there exists a set $\{\xi^j\}_{j=1}^m\subset \Omega$ of $m \le C_1 K^2 N$ points such that for any $f\in X_N$  we have
\be\label{conddisc}
\|f\|_\infty \le C_2^{-1/2}KN^{1/2}\left(\frac{1}{m}\sum_{j=1}^m |f(\xi^j)|^2\right)^{1/2} \le C_2^{-1/2}KN^{1/2} \max_{1\le j\le m} |f(\xi^j)|,
\ee
and
\be\label{LTD1}
C_2 \|f\|_{L_2(\Omega,\mu)}^2 \le \frac{1}{m}\sum_{j=1}^m |f(\xi^j)|^2 \le C_3 K^2\|f\|_{L_2(\Omega,\mu)}^2.
\ee
\end{Theorem}

The above Theorems~\ref{BP1} and~\ref{BP1a} are conditional results -- they hold under
the Nikol'skii inequality assumption. Recently, it was understood (see \cite{KPUU2}) that
some of those conditional results can be converted into unconditional ones. This progress is
based on the  result of J. Kiefer and J. Wolfowitz \cite{KW}.

David Krieg pointed out to the second author that there is a known result of J. Kiefer and J. Wolfowitz \cite{KW}
(see also \cite[Section 2]{KPUU2}) on the Nikol'skii type inequalities.

\begin{Theorem}[{\cite{KW}}]\label{KWTh} Let $X_N$ be
a finite-dimensional real subspace of $\cC(\Omega)$. Then
there exists a probability measure $\mu$ on $\Omega$ such that for all
$f\in X_N$ we have
\be\label{KW}
\|f\|_\infty \le \sqrt{N}\|f\|_{L_2(\Omega,\mu)}.
\ee
\end{Theorem}

This and Theorem \ref{BP1} imply the following statement.

\begin{Corollary}[{\cite{KoTe}}]\label{AC2}
For  any real $X_N$, there exists a set of $m\leq C_1N$ points $\xi^1,\ldots, \xi^m\in\Omega$ such that for any $f\in X_N$,
$$
\|f\|_\infty \le C_2\sqrt{N} \max_j |f(\xi^j)|,
$$
where $C_1, C_2>0$ are two absolute constants.
\end{Corollary}

Thus, in the case of real subspaces,  Corollary \ref{AC2} is a direct corollary of known results
(Theorems \ref{BP1} and  \ref{KWTh}). A standard simple argument derives from (\ref{KW}) a similar
inequality in the complex case with $\sqrt{N}$ replaced by $\sqrt{2N}$, which makes
Corollary \ref{AC2}  hold in the complex case as well. The authors of \cite{KPUU2}
proved (\ref{KW}) in the complex case, which required a non-trivial argument. Also, a statement  stronger than
Corollary \ref{AC2}  is proved in \cite{KPUU2}. In particular, they replaced $C_1N$ by $2N$ and continuous functions by bounded functions.

\subsection{Universal discretization}
\label{disud}

We now formulate two results on the universal discretization, which are used in this paper. Theorem \ref{IT2} addresses the case $p \in (2, \infty)$, while Theorem \ref{RT2} covers $p \in [1, 2]$. For the detailed discussion of the universal discretization results we refer the reader to the very recent survey paper \cite{DTdiscsurv}.

	\begin{Theorem}[{\cite{DT}}]\label{IT2}   Assume that $\cD_N=\{\ff_j\}_{j=1}^N$ is a uniformly bounded Riesz system  satisfying (\ref{ub}) and \eqref{Riesz} for some constants $0<R_1\leq R_2<\infty$.
		Let $2<p<\infty$ and let  $1\leq u\leq N$ be an integer. 		
		 Then for a large enough constant $C=C(p,R_1,R_2)$ and any $\va\in (0, 1)$,   there exist
		$m$ points  $\xi^1,\cdots, \xi^m\in  \Omega$  with
\be\label{m}
		m\leq C\va^{-7}       u^{p/2} (\log N)^2,
\ee
			such that for any $f\in  \Sigma_u(\cD_N)$,
		\[ (1-\va) \|f\|_p^p \leq \frac   1m \sum_{j=1}^m |f(\xi^j)|^p\leq (1+\va) \|f\|_p^p. \]	
	\end{Theorem}

  \begin{Theorem}[{\cite{DTM2}}]\label{RT2} Let $1\le p\le 2$. Assume that $ \cD_N=\{\ff_j\}_{j=1}^N\subset \cC(\Omega)$ is a  system  satisfying  conditions  \eqref{ub} and   \eqref{Bessel} for some constant $K\ge1$. Let $\xi^1,\cdots, \xi^m$ be independent
 	random points on $\Omega$  that are  identically distributed  according to  $\mu$.
 	 Then there exist constants  $C=C(p)>1$ and $c=c(p)>0$ such that
 	  given any   integers  $1\leq u\leq N$ and
 	 $$
 	 m \ge  C Ku \log N\cdot (\log(2Ku ))^2\cdot (\log (2Ku )+\log\log N),
 	 $$
 	 the inequalities
 	 \begin{equation}\label{Ex2}
 	 \frac{1}{2}\|f\|_p^p \le \frac{1}{m}\sum_{j=1}^m |f(\xi^j)|^p \le \frac{3}{2}\|f\|_p^p,\   \   \ \forall f\in  \Sigma_u(\cD_N)
 	 \end{equation}
 hold with probability $\ge 1-2 \exp\left( -\frac {cm}{Ku\log^2 (2Ku)}\right)$.
\end{Theorem}

\section{Recovery by $\ell_p$ minimization}
\label{lp}

\subsection{Projections on subspaces}
\label{ps}

Let $X_N$ be an $N$-dimensional subspace of the space of continuous functions $\cC(\Omega)$. As  in Section \ref{RILi}, for a fixed $m$ and a set of $m$ points  $\xi:=\{\xi^\nu\}_{\nu=1}^m\subset \Omega$,  we associate with a function $f\in \cC(\Omega)$ a vector
$$
S(f,\xi) := (f(\xi^1),\dots,f(\xi^m)) \in \bbC^m,
$$
and consider the weighted $\ell_p$ norm
$$
\|S(f,\xi)\|_{p,\bw}:= \left(\sum_{\nu=1}^m w_\nu |f(\xi^\nu)|^p\right)^{1/p},\quad 1\le p<\infty,
$$
with the usual change when $p=\infty$.
Define the best approximation of $f\in L_p(\Omega,\mu)$, $1\le p\le \infty$, by elements of $X_N$ as follows
$$
d(f,X_N)_p := \inf_{g\in X_N} \|f-g\|_p.
$$
It is well known that there exists an element, which we denote by $P_{X_N,p}(f)\in X_N$, such that
$$
\|f-P_{X_N,p}(f)\|_p = d(f,X_N)_p.
$$
The mapping $P_{X_N,p}: L_p(\Omega,\mu) \to X_N$, which may not be unique (in the case of $p=1$ or $p=\infty$), is called the Chebyshev projection.

Following \cite{VT183}, we will prove Theorem \ref{AT1} below under   assumptions {\bf A1} and {\bf A2}.

{\bf A1. Discretization.} Let $1\le p\le \infty$. Suppose that $\xi:=\{\xi^j\}_{j=1}^m\subset \Omega$ is such that for any
$g\in X_N$
$$
\|g\|_p \le D\|S(g,\xi)\|_{p,\bw}
$$
with a positive constant $D$ which may depend on $d$ and $p$.

{\bf A2. Weights.} Suppose that there is a positive constant $W=C(p)$ such that
$\sum_{\nu=1}^m w_\nu \le W$.

Consider the following well known recovery operator (algorithm) (see Algorithm 1 in Section \ref{RI})
$$
\ell p\bw(\xi)(f) := \ell p\bw(\xi,X_N)(f):=\underset{u\in X_N}  {\operatorname{argmin}}\, \|S(f-u,\xi)\|_{p,\bw}.
$$
Note that the above algorithm $\ell p\bw(\xi)$ only uses the function values $f(\xi^\nu)$, $\nu=1,\dots,m$. In the case $p=2$ it is a linear algorithm -- orthogonal projection with respect
to the norm $\|\cdot\|_{2,\bw}$. Therefore, in the case $p=2$ approximation error by the algorithm $\ell 2\bw(\xi)$ gives an upper bound for the recovery characteristic $\ro_m(\cdot, L_2)$. In the case $p\neq 2$ approximation error by the algorithm $\ell p\bw(\xi)$ gives an upper bound for the recovery characteristic $\ro_m^o(\cdot, L_p)$.

\begin{Theorem}[{\cite{VT183}}]\label{AT1} Let  $1\le p<\infty$. Then
 under assumptions {\bf A1} and {\bf A2},  for any $f\in \cC(\Omega)$,  we have
$$
\|f-\ell p\bw(\xi)(f)\|_p \le (2DW^{1/p} +1)d(f, X_N)_\infty.
$$
Furthermore, under assumption {\bf A1} with $p=\infty$,  for any $f\in \cC(\Omega)$,  we have
$$
\|f-\ell \infty(\xi)(f)\|_\infty \le (2D +1)d(f, X_N)_\infty.
$$
\end{Theorem}
\begin{proof} We give a detailed proof for the case $1\le p<\infty$. The case $p=\infty$ is similar and even simpler. For brevity denote $g:=P_{X_N,\infty}(f)$.
 From the definition of the mapping $P_{X_N,\infty}$,  we obtain
\be\label{A1}
\|f-g\|_p \le \|f-g\|_\infty = d(f, X_N)_\infty.
\ee
Clearly,
$$
\|S(f-g,\xi)\|_\infty \le \|f-g\|_\infty = d(f, X_N)_\infty.
$$
Therefore, by {\bf A2} we get
\be\label{A2}
\|S(f-g,\xi)\|_{p,\bw} \le W^{1/p} \|S(f-g,\xi)\|_\infty \le W^{1/p}d(f, X_N)_\infty.
\ee
Next, by the definition of the algorithm $\ell p\bw(\xi)$ and by \eqref{A2}, we obtain
\be\label{A3}
\|S(f-\ell p\bw(\xi)(f),\xi)\|_{p,\bw} \le \|S(f-g,\xi)\|_{p,\bw} \le W^{1/p}d(f, X_N)_\infty.
\ee
Bounds (\ref{A2}) and (\ref{A3}) imply
\be\label{A4}
\|S(g-\ell p\bw(\xi)(f),\xi)\|_{p,\bw} \le 2W^{1/p}d(f, X_N)_\infty.
\ee
Then, the discretization assumption {\bf A1} implies
\be\label{A5}
\|g-\ell p\bw(\xi)(f)\|_{p} \le 2D W^{1/p}d(f, X_N)_\infty.
\ee
Combining bounds (\ref{A1}) and (\ref{A5}), we conclude
$$
\|f-\ell p\bw(\xi)(f)\|_p \le \|f-g\|_p +\|g-\ell p\bw(\xi)(f)\|_{p}\le (1+2DW^{1/p})d(f, X_N)_\infty,
$$
which completes the proof of Theorem \ref{AT1}.
\end{proof}

The following version of Theorem~\ref{AT1} for the error of $\|f-\ell p\bw(\xi)(f)\|_{\infty}$ under an extra condition on the Nikol'skii inequality was proved in \cite{LMT}.

\begin{Theorem}[{\cite{LMT}}]\label{BT1a} Let $1\le p<\infty$.  Under assumptions {\bf A1}, {\bf A2}, and an extra assumption $X_N\in NI(p,\infty,H)$   for any $f\in \cC(\Omega)$ we have
 $$
\|f-\ell p\bw(\xi)(f)\|_{\infty} \le (2HD W^{1/p} +1)d(f, X_N)_\infty.
$$

\end{Theorem}

\begin{proof} The proof is simple and goes along the lines of the proof of Theorem~\ref{AT1}. Let $u:= \ell p\bw(\xi)(f)$. For an arbitrary $g\in X_N$, we have the following chain of inequalities
\begin{align*}
\|f-u\|_\infty &\le \|f-g\|_\infty + \|g-u\|_\infty \le \|f-g\|_\infty +H\|g-u\|_p\\
&\le  \|f-g\|_\infty + HD\|S(g-u,\xi)\|_{p,\bw}\\
&\le  \|f-g\|_\infty + HD(\|S(f-g,\xi)\|_{p,\bw}+ \|S(f-u,\xi)\|_{p,\bw})\\
&\le  \|f-g\|_\infty + 2HD\|S(f-g,\xi)\|_{p,\bw}\\
&
\le  \|f-g\|_\infty + 2HDW^{1/p}\|S(f-g,\xi)\|_{\infty} \le (1+ 2HDW^{1/p})\|f-g\|_\infty.
\end{align*}
Minimizing over all $g\in X_N$, we complete the proof.

\end{proof}

The following analog of Theorem~\ref{AT1} with a weaker discretization assumption was proved in \cite{LMT}.

\begin{Theorem}[{\cite[Theorem~5.2]{LMT}}]\label{BT2} Let $p\in [1,\infty)$. Assume that a subspace
    $X_N\subset \cC(\Omega)$ has the property $X_N\in\LD(m,p,\infty,D)$ that is realized
    by a set $\xi=\{\xi^j\}_{j=1}^m$ of $m$ points in $\Omega$; namely,
\be\label{B1}
\|g\|_p \le D\max_{1\le j\le m} |g(\xi^j)|,\  \ \forall g\in X_N.
\ee
Then for any $f\in \cC(\Omega)$,  we have
$$
\|f-\ell \infty(\xi)(f)\|_p \le (2D  +1)d(f, X_N)_\infty.
$$
\end{Theorem}

We now make some comments on the above results. Theorem \ref{AT1} guarantees that under assumptions {\bf A1} and {\bf A2},  we have,  for $p\in [1,\infty)$,
\be\label{ps1}
\|\ell p\bw(\xi)(f)\|_p \le (2DW^{1/p} +2)\|f\|_\infty.
\ee
If in addition we assume that $X_N \in NI(p,\infty,H)$, then we obtain
\be\label{ps2}
\|\ell p\bw(\xi)(f)\|_\infty \le H(2DW^{1/p} +2)\|f\|_\infty.
\ee

Thus, we can formulate the following direct corollary of Theorem \ref{AT1}.

\begin{Corollary}\label{psC1} Assume that $X_N\subset \cC(\Omega)$  satisfies assumptions {\bf A1}, {\bf A2} and, in addition,
$X_N \in NI(p,\infty,H)$ with $\|\cdot\|_p := \|\cdot\|_{L_p(\Omega,\mu)}$. Then we have
\be\label{ps3}
\|\ell p\bw(\xi,X_N)(f)\|_\infty \le H(2DW^{1/p} +2)\|f\|_\infty,\qquad \forall f \in \cC(\Omega).
\ee
\end{Corollary}

Recently, the problem of finding good bounds on norms of projectors, acting from $\cC(\Omega)$ to $\cC(\Omega)$, was discussed in \cite{KPUU2}. In particular, projectors, which only use sampling points, were discussed there.
We now present some corollaries of the above results for that setting.
Corollary \ref{psC1} gives a conditional result in the case $p=2$, when the operator $\ell p\bw(\xi,X_N)$ is the linear projector operator. Let us use the known result and make
it unconditional. Take any subspace $X_N\subset \cC(\Omega)$. Set $p=2$ and
apply Theorem \ref{KWTh}. It provides us the probability measure $\mu$ such that
$X_N \in NI(2,\infty,N^{1/2})$. Next, we apply Theorem \ref{LimTT} with $K=1$ to
the subspace $X_N$ and the norm $\|\cdot\|_2 := \|\cdot\|_{L_2(\Omega,\mu)}$. It gives us a set $\{\xi^j\}_{j=1}^m$, $m\le C_1N$, with absolute constant $C_1$ such that assumption {\bf A1} is satisfied with equal weights $\bw_m:=(1/m,\dots,1/m)$.
Therefore, assumption {\bf A2} is satisfied with $W=1$. Finally, we apply Corollary \ref{psC1} and obtain
\be\label{ps4}
\|\ell 2\bw_m(\xi,X_N)\|_{L_\infty\to L_\infty} \le CN^{1/2},
\ee
with an absolute constant $C$.

Note that in addition to (\ref{ps4}), we have the following useful discretization property, which follows from Theorem \ref{LimTT}: For any $f\in X_N$
$$
C_2 \|f\|_2^2 \le \frac{1}{m}\sum_{j=1}^m |f(\xi^j)|^2 \le C_3 \|f\|_2^2,
$$
where $C_2$ and $C_3$ are absolute positive constants.

We can improve the above bound $m\le C_1N$ by a stronger one $m \le bN$, where $b$ is any fixed number from $(1,2]$, if instead of  Theorem \ref{LimTT},
we use Theorem \ref{disdT1}. In this case instead of the equal weights $\bw_m$ we
need to use the weights $\bw$ provided by Theorem \ref{disdT1}.
Note that the authors of \cite{KPUU2} proved an analog of (\ref{ps4}) with $C_1=2$.

The following simple fact is a classical result in approximation theory, which is often used. Let $P\, : \, X\, \to\, X$ be a linear projector on $X_N$ with the property
$\|P\|_{X\to X} \le L$. Then for any $f\in X$,  we have the Lebesgue inequality
\be\label{ps5}
\|f-P(f)\|_X \le (L+1)d(f,X_N)_X.
\ee
Indeed, for any $g\in X_N$, we have
$$
\|f-P(f)\|_X = \|f-g-P(f-g)\|_X \le (L+1)\|f-g\|_X,
$$
which implies (\ref{ps5}).

Thus, (\ref{ps4}) implies that for any $f\in \cC(\Omega)$,  we have
\be\label{ps6}
\|f-\ell 2\bw_m(\xi,X_N)(f)\|_\infty \le (CN^{1/2}+1)d(f,X_N)_\infty.
\ee

Algorithms based on the $\ell_p$ minimization play important role in the sampling recovery
theory. For this reason, the following optimal characteristics were introduced and studied
(see \cite{VT183} for $p=2$).
For a fixed $m$,   define,  for a class $\bF$ of functions on $\Omega$,
$$
\varrho_m^{wls}(\bF,L_p) := \inf_{\bw,\,\xi\subset \Omega,\,X_N} \sup_{f\in \bF} \|f-\ell p \bw(\xi,X_N)(f)\|_p.
$$
The following result was obtained in \cite{VT183}.

\begin{Theorem}[{\cite{VT183}}]\label{VT183T2} There exist two positive absolute constants $b$ and $B$ such that for any   compact subset $\Omega$  of $\bbR^d$, any probability measure $\mu$ on it, and any compact subset $\bF$ of $\cC(\Omega)$,  we have
$$
\ro_{bn}^{wls}(\bF,L_2(\Omega,\mu)) \le Bd_n(\bF,L_\infty).
$$
\end{Theorem}

The above theorem is devoted to recovery by weighted least squares algorithms $\ell 2 \bw(\xi)$. We may want to have the recovery algorithm $\ell 2 \bw(\xi)$ to be a classical least squares algorithm, i.e. $\bw=\bw_m:=(1/m,\dots,1/m)$. In this case we introduce the corresponding optimal characteristics as follows (see \cite{VT183})
$$
\varrho_m^{ls}(\bF,L_2) := \inf_{\xi,\,X_N} \sup_{f\in \bF} \|f-\ell 2 \bw_m(\xi,X_N)(f)\|_2
$$
and (see \cite{VT189})
$$
\varrho_m^{lp}(\bF,L_p) := \inf_{\xi,\,X_N} \sup_{f\in \bF} \|f-\ell p \bw_m(\xi,X_N)(f)\|_p.
$$

The reader can find analogs of Theorem \ref{VT183T2} for $\varrho_m^{lp}(\bF,L_p)$ with
the Kolmogorov width $d_n(\bF,L_\infty)$ replaced by the $E$-conditioned Kolmogorov width (see \cite{VT183} and \cite{VT189})
$$
d_N^{E}(\bF,L_p) :=     \inf_{\{u_1,\dots,u_N\}\, \text{satisfies Condition} E}   \sup_{f\in \bF}\inf_{c_1,\dots,c_N}\left\|f-\sum_{i=1}^N c_iu_i\right\|_p
$$
for different conditions $E$.

\subsection{Sparse approximations}

 For brevity denote $L_p(\xi) := L_p(\Omega_m,\mu_m)$, where $\Omega_m=\{\xi^\nu\}_{\nu=1}^m$  and  $\mu_m(\{\xi^\nu\}) =1/m$, $\nu=1,\dots,m$. Let
$B_v(f,\cD_N,L_p(\xi))$ (see Algorithm 2 in Section \ref{RI}) denote the best $v$-term approximation of $f$ in the $L_p(\xi)$ norm with
respect to the system $\cD_N$. Note that $B_v(f,\cD_N,L_p(\xi))$ may not be unique. Obviously,
\be\label{D5}
\|f-B_v(f,\cD_N,L_p(\xi))\|_{L_p(\xi)} = \sigma_v(f,\cD_N)_{L_p(\xi)}.
\ee

We proved in \cite{DTM3} the following theorem.

 \begin{Theorem}[{\cite{DTM3}}]\label{NUT2} Let $1\le p<\infty$ and let $m$, $v$, $N$ be given natural numbers such that $2v\le N$.  Let $\cD_N\subset \cC(\Omega)$ be  a system of $N$ elements. Assume that  there exists a set $\xi:= \{\xi^j\}_{j=1}^m \subset \Omega $, which provides  the  one-sided $L_p$-universal sampling discretization (see Definition \ref{ID1a})
  \be\label{D6}
 \|f\|_p \le D\left(\frac{1}{m} \sum_{j=1}^m |f(\xi^j)|^p\right)^{1/p}, \quad \forall\, f\in \Sigma_{2v}(\cD_N).
\ee
   Then, for   any  function $ f \in \cC(\Omega)$, we have
\be\label{D7}
  \|f-B_v(f,\cD_N,L_p(\xi))\|_{L_p(\Omega,\mu)} \le 2^{1/p}(2D +1) \sigma_v(f,\cD_N)_{L_p(\Omega, \mu_\xi)}
 \ee
 and
 \be\label{D8}
  \|f-B_v(f,\cD_N,L_p(\xi))\|_{L_p(\Omega,\mu)} \le  (2D +1) \sigma_v(f,\cD_N)_\infty.
 \ee
 \end{Theorem}

The following theorem was proved in \cite{VT203}.

 \begin{Theorem}[{\cite[Theorem 3.3]{VT203}}]\label{NUT3} Let $2\le p<\infty$ and let $m$, $v$, $N$ be given natural numbers such that $2v\le N$.  Let $\cD_N\subset \cC(\Omega)$ be  a system of $N$ elements such that $\cD_N \in NI(2,p,H,2v)$. Assume that  there exists a set $\xi:= \{\xi^j\}_{j=1}^m \subset \Omega $, which provides the one-sided $L_2$-universal sampling discretization
  \be\label{D12}
 \|f\|_2 \le D\left(\frac{1}{m} \sum_{j=1}^m |f(\xi^j)|^2\right)^{1/2}, \quad \forall\, f\in \Sigma_{2v}(\cD_N).
\ee
   Then, for   any  function $ f \in \cC(\Omega)$,  we have (see the definition of $\mu_\xi$ in (\ref{muxi}))
\be\label{D13}
  \|f-B_v(f,\cD_N,L_2(\xi))\|_{L_p(\Omega,\mu)} \le 2^{1/p}(2DH +1) \sigma_v(f,\cD_N)_{L_p(\Omega, \mu_\xi)}
 \ee
 and
 \be\label{D14}
  \|f-B_v(f,\cD_N,L_2(\xi))\|_{L_p(\Omega,\mu)} \le (2DH +1) \sigma_v(f,\cD_N)_\infty.
 \ee

 \end{Theorem}

The following Theorem \ref{ubT5} from \cite{LMT} establishes (\ref{D8}) under somewhat weaker conditions.

\begin{Theorem}[{\cite{LMT}}]\label{ubT5} Let $p\in [1,\infty)$ and
 let $m$, $v$, $N$ be given natural numbers such
 that $2v\le N$.  Let $\cD_N\subset \cC(\Omega)$ be  a system of $N$ elements.
 Assume that  there exists a set $\xi:= \{\xi^j\}_{j=1}^m \subset \Omega $,
     which provides universal LDI$(p,\infty$), (\ref{I3}), for the collection $\cX_{2v}(\cD_N)$. Then for   any  function $
 f \in \cC(\Omega)$ we have
 \be\label{ub17}
  \|f - B_v(f,\cD_N,L_\infty(\xi))\|_p \le  (2D +1) \sigma_v(f,\cD_N)_\infty.
 \ee
 \end{Theorem}
 \begin{proof} We derive  (\ref{ub17}) from the following obvious relation:
 \be\label{ub15}
 \|f-B_v(f,\cD_N,L_\infty(\xi))\|_{L_\infty(\xi)} = \sigma_v(f,\cD_N )_{L_\infty(\xi)}.
 \ee
  Clearly,
$$
\sigma_v(f,\cD_N)_{L_\infty(\xi)} \le \sigma_v(f,\cD_N )_\infty.
$$
For brevity denote $u:=B_v(f,\cD_N,L_\infty(\xi))$ and let $h$ be the best
     $L_\infty$-approximation to $f$ from $\Sigma_v(\cD_N)$.   Then (\ref{ub15}) implies
$$
\|h - u\|_{L_\infty(\xi)} \le \|f-h\|_{L_\infty(\xi)} +\|f-u\|_{L_\infty(\xi)}
 \le 2\sigma_v(f,\cD_N)_\infty.
$$
Using that $h - u \in \Sigma_{2v}(\cD_N)$, by discretization (\ref{I3}),  we
conclude that
\be\label{ub18}
\|h - u\|_{L_p(\Omega,\mu)} \le 2D \sigma_v(f,\cD_N)_\infty.
\ee
Finally,
$$
\|f-u\|_{L_p(\Omega,\mu)} \le \|f-h\|_{L_p(\Omega,\mu)} + \|h - u\|_{L_p(\Omega,\mu)}.
$$
This and (\ref{ub18}) prove (\ref{ub17}).

\end{proof}

The above discussed algorithms $B_v(f,\cD_N,L_p(\xi))$ only use information on the sampling vector $S(f,\xi)$ and, therefore, they can be used for proving the upper bounds
for $\rho_m^o(\bF,L_p)$. The algorithms $B_v(f,\cD_N,L_p(\xi))$ can be defined as follows (see (\ref{Alg2}))
$$
L^s(\xi,f) :=   \underset {L\in \cX_v(\cD_N)} {\operatorname{argmin}} \|S(f-\ell p(\xi,L)(f),\xi)\|_p,
$$
\be\label{Alg2a}
  B_v(f,\cD_N,L_p(\xi)):= \ell p(\xi,L^s(\xi,f))(f).
\ee
The following analog of this algorithm was studied in \cite{DTM1} and \cite{VT198}. We give its definition in somewhat more general terms.
We define this algorithm for a collection $\cX:= \{X(n)\}_{n=1}^k$ of finite-dimensional subspaces as follows:
$$
n(\xi,f) :=   \underset {1\leq n\leq k} {\operatorname{argmin}}\|f-\ell p(\xi,X(n))(f)\|_p,
$$
\be\label{I4}
  \ell p(\xi,\cX)(f):= \ell p(\xi,X(n(\xi,f)))(f).
\ee

The following result was proved in \cite{DTM3}.

 \begin{Theorem}[{\cite{DTM3}}]\label{ubT3} Let $m$, $v$, $N$ be given positive integers  such that $v\le N$.  Let  $\cX:= \{X(n)\}_{n=1}^k$ be a collection of finite-dimensional subspaces.   Assume that  there exists a set $\xi:= \{\xi^j\}_{j=1}^m  $ of $m$ points in $\Omega$,  which provides the {\it one-sided $L_p$-universal sampling discretization}
   \be\label{D6ub}
 \|f\|_p \le D\left(\frac{1}{m} \sum_{j=1}^m |f(\xi^j)|^p\right)^{1/p}, \quad \forall\, f\in \bigcup_{n=1}^k X(n),
\ee
  for the collection $\cX$.
   Then for   any  function $ f \in \cC(\Omega)$,  we have
\be\label{I5}
  \|f-\ell p(\xi,\cX)(f)\|_p \le 2^{1/p}(2D +1) \min_{1\le n\le k}d(f, X(n))_{L_p(\Omega, \mu_\xi)}
 \ee
 and
 \be\label{I6}
  \|f-\ell p(\xi,\cX)(f)\|_p \le  (2D +1) \min_{1\le n\le k}d(f, X(n))_{\infty}.
 \ee
 \end{Theorem}

 {\bf Comment \ref{lp}.1.} First of all, we note that the algorithm $B_v(\cdot,\cD_N,L_p(\xi))$ only uses
the function values at points $\xi$ and the algorithm $\ell p(\xi,\cX)$ uses an extra information for choosing the $n(\xi,f)$. Second, Theorems \ref{NUT2} and \ref{ubT3} (in the case $\cX=\cX_v(\cD_N)$) are very similar: We require one-sided universal discretization
for $\Sigma_{2v}(\cD_N)$ in Theorem \ref{NUT2} and require one-sided universal discretization for $\Sigma_{v}(\cD_N)$ in Theorem \ref{ubT3} (in the case $\cX=\cX_v(\cD_N)$). However, this makes a big difference in the case of general collections
$\cX$. It is demonstrated in \cite{VT198} on the example of a collection of subspaces of the multivariate trigonometric polynomials with frequencies from  parallelepipeds of approximately the same volume. The recovery results are obtained in \cite{VT198} for the algorithm $\ell \infty(\xi,\cX)$ but not for an analog of the algorithm $B_v(\cdot,\cD_N,L_\infty(\xi))$.

\section{Recovery by greedy algorithms}
\label{gr}

\subsection{WOMP}\label{womp}

The Weak Orthogonal Matching Pursuit (WOMP) is a greedy algorithm defined with respect to a dictionary (or system) $\mathcal{D} \subset H$ in a Hilbert space $H$, which is equipped with the inner product $\langle \cdot, \cdot \rangle$ and the induced norm $\|\cdot\|_H$. This algorithm is also frequently referred to as the Weak Orthogonal Greedy Algorithm (see, e.g., \cite{VTbook}).\\

{\bf Weak Orthogonal Matching Pursuit (WOMP).} Let $\cD=\{g\}$ be a system of  nonzero elements in a Hilbert space  $H$ such that $\sup_{g\in\cD}\|g\|_H\le 1$.
 Let $\tau:=\{t_k\}_{k=1}^\infty\subset (0, 1]$ be a given  sequence of weakness parameters.
Given  $f_0\in H$, we define  a sequence  $\{f_k\}_{k=1}^\infty\subset H$ of residuals  for $k=1,2,\cdots$  inductively  as follows:
\begin{enumerate}[\rm (1)]
	\item
 $ g_k\in \cD$  is any  element  satisfying
$$
|\langle f_{k-1},g_k\rangle | \ge t_k
\sup_{g\in \cD} |\langle f_{k-1},g\rangle |.
$$

\item  Let  $H_k := \sp \{g_1,\dots,g_k\}$, and define
$G_k^\tau(\cdot, \cD)$ to be the orthogonal projection operator from $H$   onto the space $H_k$ .

\item   Define the residual after the $k$th iteration of the algorithm by
\begin{equation*}
f_k := f_0-G_k^\tau(f_0, \cD).
\end{equation*}
\end{enumerate}

When $t_k=1$ for all $k \ge 1$, the WOMP reduces to the Orthogonal Matching Pursuit (OMP). In this paper, we restrict our focus to the case where $t_k = t \in (0, 1]$ for all $k \ge 1$. We denote by $\{G_k^t(\cdot, \mathcal{D})_H\}_{k=1}^\infty$ the WOMP  defined with respect to the weakness parameter $t$ and the system $\mathcal{D}$ in a Hilbert space $H$.
While the elements $g_k$ selected in each step of the algorithm may not be unique, all results presented below are independent of the specific choice of $g_k$. For a comprehensive treatment of greedy approximation theory, we refer the reader to \cite{VTbook}.

 {\bf UP($u,D$). ($u,D$)-unconditional property.}  Let $u, D$ be integers such that $1\leq u\leq D$.
  We say that a system $\cD=\{\vi_i\}_{i\in I}$ of  elements
 in a Hilbert space $H=(H, \|\cdot\|)$
  is ($u,D$)-unconditional with  constant $U>0$ if for any $A\subset I$ and $J\subset I\setminus A$ such that $|A|\leq u$ and  $|A|+|J| \le D$, and for any $f=\sum_{i\in A} c_i \vi_i\in \Sigma_u(\cD)$,  we have
\be\label{UP}
\|f\|\leq U\, \inf_{g\in V_J(\cD)} \|f-g\|,
\ee
where $ V_J(\cD):=\sp\{\vi_i:\  \ i\in J\}$.


 \begin{Theorem}[{\cite[Corollary I.1]{LT}}]\label{PT1} Let $\cD$ be a dictionary in a Hilbert space $H=(H, \|\cdot\|)$ having  the property  {\bf UP($u,D$)}  with   constant $U>0$, where  $u, D$ are integers such that  $1\leq u\leq D$.   Let $f\in H$, and  $t\in (0, 1]$.
  Then there exists a positive constant $c_\ast:=c(t,U)$  depending only on $t$ and  $U$ such that for every integer
   \[ 1\leq v \leq \min\Big\{ u, \  \frac D {1+c_\ast}\Big\},\]
   we have
   \[ \sigma_{\left \lceil{c_\ast v}\right\rceil}(f, \cD)_H \leq \Big\| f- G_{\left \lceil{c_\ast v}\right\rceil}^t(f, \cD)_H\Big\|\le C\sigma_v(f_0,\cD)_H,\]
   where  $C>1$ is an absolute constant, and $\{G_k^t(\cdot, \mathcal{D})_H\}_{k=1}^\infty$ denotes the WOMP  defined with respect to the weakness parameter $t$ and the system $\mathcal{D}$ in  $H$.

\end{Theorem}

  We will  consider the Hilbert space $L_2(\Omega_m,\mu_m)$
 instead of  $L_2(\Omega,\mu)$,
where $\Omega_m=\{\xi^\nu\}_{\nu=1}^m$ is a set of points that provides a good universal discretization,  and $\mu_m$ is the uniform probability measure on $\Omega_m$, i.e.,
$\mu_m\{ \xi^\nu\} =1/m$, $\nu=1,\dots,m$.    Let $\cD_N(\Omega_m)$ denote  the restriction
of a system  $\cD_N$ on the set  $\Omega_m$.
Theorem \ref{PT2} below  guarantees that the simple greedy algorithm WOMP gives the corresponding Lebesgue-type inequality in the norm $L_2(\Omega_m,\mu_m)$, and hence
 provides good sparse recovery. For simplicity, we denote $L_2(\Omega_m,\mu_m)$ as $L_2^m$ when the underlying set $\Omega_m$ is clear from the context.
 Theorem \ref{PT2} was derived in \cite{DTM2} from Theorem \ref{PT1}.

\begin{Theorem}[{\cite{DTM2}}]\label{PT2}  Let  $\cD_N=\{\vi_j\}_{j=1}^N$ be  a uniformly bounded Riesz basis  in $L_2(\Omega, \mu)$  satisfying  \eqref{ub} and \eqref{Riesz} for some constants $0<R_1\leq R_2<\infty$.
Let $\Omega_m=\{\xi^1,\cdots, \xi^m\}$ be a finite subset of $\Omega$  that provides the $L_2$-universal sampling discretization (\ref{ud}) for the collection
$\cX_u(\cD_N)$ and a given integer $1\leq u\leq N$.    Then  given each weakness parameter $t\in (0, 1]$,   there exists a constant   $c_\ast=c(t,R_1,R_2)\ge 1$
depending only on  $t$ and the constants $R_1$ and $R_2$
 such that for any integer $0\leq v\leq u/(1+c_\ast)$ and any $f\in \cC(\Omega)$,
 we have
 \be\label{mp}
\sigma_{\lceil c_\ast v\rceil}(f,\cD_N)_{L_2^m}\leq \Big\|f-G_{\lceil c_\ast v\rceil}^t(f,\cD)_{L_2^m}  \Big\|_{L_2^m} \le C\sigma_v(f,\cD_N)_{L_2^m},
\ee
and
\be\label{mp2}
\Big\|f-G_{\lceil c_\ast v\rceil}^t(f,\cD)_{L_2^m}  \Big\|_{L_2(\Omega,\mu)}
 \le C\sigma_v(f,\cD_N)_{L_\infty(\Omega,\mu)},
\ee
 where $C>1$ is an  absolute constant,
  and $\{G_k^t(\cdot, \mathcal{D})_{L_2^m}\}_{k=1}^\infty$ denotes the WOMP  defined with respect to the weakness parameter $t$ and the discrete system $\mathcal{D}(\Omega_m)$ in  the Hilbert space $L_2^m=L_2(\Omega_m,\mu_m)$.

 \end{Theorem}

 Here is an extension of Theorem \ref{PT2} to the case of $L_p(\Omega,\mu)$, $p\in [2,\infty)$, under an extra Nikol'skii-type inequality condition.

 \begin{Theorem}[{\cite[Theorem 3.1]{VT203}}]\label{NUT1}  Let  $\cD_N=\{\vi_j\}_{j=1}^N$ be  a uniformly bounded Riesz system  in $L_2(\Omega, \mu)$  satisfying  \eqref{ub} and \eqref{Riesz} for some constants $0<R_1\leq R_2<\infty$.
Let $\Omega_m=\{\xi^1,\cdots, \xi^m\}$ be a finite subset of $\Omega$  that provides the $L_2$-universal sampling discretization for the collection
$\cX_u(\cD_N)$ with  $1\leq u\leq N$.   Assume in addition that $\cD_N \in NI(2,p,H,u)$ for some $p\in [2,\infty)$. Then, for a given  weakness parameter $0<t\leq 1$,   there exists a constant integer  $c_\ast=c(t,R_1,R_2)\ge 1$ depending only on $t$ and the constants $R_1$ and $R_2$  such that for any integer $0\leq v\leq u/(1+c_\ast)$ and any $f\in \cC(\Omega)$,  we have
 \be\label{mpn}
\sigma_{\lceil c_\ast v\rceil}(f,\cD_N)_{L_2^m}\leq \Big\|f-G_{\lceil c_\ast v\rceil}^t(f,\cD)_{L_2^m}  \Big\|_{L_2^m} \le C\sigma_v(f,\cD_N)_{L_2^m},
\ee
and
\be\label{mp3n}
\Big\|f-G_{\lceil c_\ast v\rceil}^t(f,\cD)_{L_2^m} \Big\|_{L_p(\Omega,\mu)}
 \le H C\sigma_v(f,\cD_N)_{L_p(\Omega,\mu_\xi)},
\ee
 where $C>1$ is an  absolute constant,
  and $\{G_k^t(\cdot, \mathcal{D})_{L_2^m}\}_{k=1}^\infty$ denotes the WOMP  defined with respect to the weakness parameter $t$ and the discrete system $\mathcal{D}(\Omega_m)$ in  the Hilbert space $L_2^m=L_2(\Omega_m,\mu_m)$, and
  $\mu_\xi$ is defined in (\ref{muxi}).


 \end{Theorem}

\subsection{WCGA}

We give the  definition of  the Weak Chebyshev Greedy Algorithm (WCGA) in a Banach space,  which  was introduced in \cite{T1}  as a generalization  of the Weak Orthogonal Matching Pursuit (WOMP).
To be more precise,
let $X^\ast$ denote the  dual of a Banach space $X$.
For a nonzero element $g\in X$,  we denote by  $F_g $  a norming (peak) functional for $g$,  that is,  an element  $F_g\in X^\ast$ satisfying
$$
\|F_g\|_{X^*} =1,\qquad F_g(g) =\|g\|_X.
$$
The existence of such a functional is guaranteed by the Hahn-Banach theorem.

Now we can define the WCGA as follows.\\

{\bf Weak Chebyshev Greedy Algorithm (WCGA).}  Let
$\tau := \{t_k\}_{k=1}^\infty$ be a given weakness sequence of  positive numbers $\leq 1$. Let $\cD=\{g\}\subset X$ be a system of nonzero elements in $X$ such that $\|g\|\le 1$ for $g\in\cD$.
Given  $f_0\in X$, we define  the elements  $f_m\in X$ and $\phi_m\in \cD$  for $m=1,2,\cdots$ inductively   as follows:

\begin{enumerate}[\rm (1)]

	\item  $\phi_m  \in \cD$ is any element satisfying
	$$
	|F_{f_{m-1}}(\phi_m)| \ge t_m\sup_{g\in\cD}  | F_{f_{m-1}}(g )|.
	$$
	
	\item  Define
	$$
	\Phi(m) := \sp \{\phi_1,\cdots, \phi_m\},
	$$
	and let $G_m(f_0, \cD)_X$  be the best approximant to $f_0$ from  the space $\Phi(m)$; that is,
	$$
	G_m(f_0, \cD)_X :=\underset {G\in \Phi(m)} {\operatorname{argmin}}\|f_0-G\|_X.
	$$
	
	\item  Define
	$$
	f_m := f_0-G_m(f_0, \cD)_X.
	$$
	
\end{enumerate}

\begin{Remark}\label{CGR1} We defined the WCGA for a system satisfying an extra condition $\|g\|\le 1$ for $g\in\cD$. Clearly, realizations of the WCGA for a new system $\cD^B := \{Bg, \, g\in \cD\}$, where $B$ is a positive number, coincide with those for the system $\cD$. This means that the restriction $\|g\|\le 1$ can be replaced by the restriction $\|g\|\le B$ with some positive number $B$.
\end{Remark}

We  will only consider the WCGA for  the case where $t_k=t\in (0, 1]$ for all $k\ge 1$. For a system $\cD\subset X$, and  a weakness parameter $t\in (0, 1]$, we denote the corresponding WCGA  as $\{G_m^t(\cdot, \mathcal{D})_X\}_{m=1}^\infty$.
We also point out that in the case when $X$ is a Hilbert space,  the WCGA coincides with the well known  WOMP, which  is very popular in signal processing, and in particular, in compressed sensing. In approximation theory the WOMP is also called the Weak Orthogonal Greedy Algorithm (WOGA).

Note that by definition,  $G_k^t(f_0, \cD)_X\in\Sigma_k(\cD)$. Consequently,
\[ \sigma_k(f_0,\cD)_X\leq \|f_k\|_X=\|f_0-G_k^t(f_0,\cD)_X\|_X.\]
Under certain conditions on the system $\cD$ and the Banach space $X$,
an inverse inequality  holds as well. This is established  in
 Theorem \ref{CGT1} below, which  was proved  for real Banach spaces in \cite{VT144} (see also \cite{VTbookMA}, Section 8.7, Theorem 8.7.17, p.431) and for complex Banach spaces in \cite{DGHKT}. Although the theorem was originally stated for dictionaries, the proof remains valid for a general system $\mathcal{D}$.

 Recall that the $(v, S)$-incoherence property of a system $\mathcal{D} \subset X$, where $v$ and $S$ are integers such that $1 \leq v \leq S$, is given in Definition \ref{ID2}. It follows directly from this definition that the $(v, S)$-incoherence property with parameters $V > 0$ and $r > 0$ implies the $(v', S)$-incoherence property with the same parameters for any integer  $1 \leq v' \leq v$.

 \begin{Theorem}[{\cite[Theorem 2.7]{VT144}, \cite[Theorem 8.7.17]{VTbookMA}}]\label{CGT1} Let $X$ be a Banach space satisfying the following condition for some parameter $1<q\le 2$ and constant $\gamma>0$:  $$\eta(X, w)\le \gamma w^q, \  \ \forall w>0.$$
 Suppose that $\cD\subset X$  is a system   having the  ($v_0,S$)-incoherence property  with parameters   $V>0$ and $r>0$, where $v_0$, $S$ are integers such that $1\leq v_0\leq S$. Let  $f\in X$, and let $t\in (0, 1]$ be a given weakness parameter.  Then there exists a constant $C(q)>0$ depending only on $q$ such that
 $$
\sigma_{u}(f,\cD)_X\leq \|f-G_{u}^t(f,\cD)_X\|_X \le C\sigma_v(f,\cD)_X
$$
for every integer $1\leq v\leq v_0$ satisfying $u+v\leq S$, where
  \[ u=u(v,t):=\left\lceil C(q)\gamma^{\frac{1}{q-1}}  \Big(\frac  {V\cdot v^r}  t\Big)^{q'}\ln (V\cdot v)\right\rceil,\   \ q':=\frac{ q}{q-1}, \]
    $C>1$ is an absolute constant, and $\{ G_k^t(\cdot,\cD)_X\}_{k=1}^\infty$ denotes the WCGA with respect to the system $\cD$ and the  weakness parameter $t$.
\end{Theorem}

We will apply Theorem \ref{CGT1} to the discretized version of the given system $\cD_N = \{g_i\}_{i=1}^N\subset \cC(\Omega)$.  Typically, instead of the space $L_p(\Omega,\mu)$, we consider the space $L_p(\Omega_m,\mu_m)$
where $\Omega_m=\{\xi^\nu\}_{\nu=1}^m$ is from Definition \ref{ID1} and  $\mu_m(\xi^\nu) =1/m$, $\nu=1,\dots,m$. Let $\cD_N(\Omega_m)$ be the restriction
of $\cD_N$ on the set $\Omega_m$. Here and elsewhere in the paper,  we  use the notation $\Omega_m$ to denote the set
$\xi$ in order to emphasize that the set $\xi$ plays the role of a new domain $\Omega_m$ consisting of $m$ points instead of
the original domain $\Omega$.

\begin{Proposition}[{\cite{VT202}}]\label{CGP1} Let $1\le p<\infty$, and let
$v, u, S$ be given integers such that $1\le v\leq u\leq S \leq N$. Assume that  $\cD_N = \{g_i\}_{i=1}^N\subset \cC(\Omega)$ is a system  that has   the    ($v,S$)-incoherence property in the space $L_p(\Og, \mu)$ for some  parameters   $V>0$ and $r>0$. Assume in addition that there exists a finite set $\Og_m:=\{\xi^1,\cdots, \xi^m\}\subset \Og$ which provides  the $L_p$-universal sampling discretization \eqref{ud} for
  the collection $\cX_u(\cD_N)$. Then the system $\cD_N(\Omega_m)$ has the ($v,u$)-incoherence property with  parameters $2^{1/p}V$ and $r$, and moreover,
\be\label{CG3}
\|g_i\|_{L_p(\Omega_m,\mu_m)} \le (3/2)^{1/p} \|g_i\|_p,\quad i=1,\dots,N.
\ee
\end{Proposition}

We now prove a conditional result on the sampling recovery by the WCGA. We call it {\it conditional} because Theorem \ref{CGT2} below is proved under two conditions on the system $\cD_N$, which are non-trivial and non-standard conditions.

\begin{Theorem}[{\cite{VT202}}]\label{CGT2} Let $1<p<\infty$,  $q:= \min(p,2)$ and  $q' := q/(q-1)$. Let
$v_0, u, S$ be given integers such that $1\le v_0\leq u\leq S \leq N$.
Let $\cD_N = \{g_i\}_{i=1}^N\subset \cC(\Omega)$ be a system that  has the ($v_0,S$)-incoherence property in the space $L_p(\Omega, \mu)$ with  parameters $V, r>0$. Assume  that there exists a finite set $\xi=\Og_m=\{\xi^1,\cdots, \xi^m\}\subset \Og$ which provides  the $L_p$-universal sampling discretization \eqref{ud} for
  the collection $\cX_u(\cD_N)$.
Then given each weakness parameter $t\in (0, 1]$, there exists a constant  $c=C(t,p)\ge 1$ depending only on $t$ and $p$ such that for any positive integer $1\leq v\leq v_0$  satisfying
 $$
  v+a(v)\leq u\   \ \text{with }\  \  a(v):=\lceil c(2V\cdot v^r)^{q'}(\ln (2Vv)) \rceil,
 $$
 and for any  $f_0\in \cC(\Omega)$, we have
 \be\label{mpc}
\sigma_{a(v)}(f_0,\cD_N)_{L_p^m}\leq \Big\|f_0-G_{a(v)}^t(f_0, \cD)_{L_p^m}\Big\|_{L_p^m} \le C\sigma_v(f_0,\cD_N)_{L_p^m},
\ee
and
\be\label{mp2c}
\Big\|f_0-G_{a(v)}^t(f_0, \cD)_{L_p^m}\Big\|_{L_p(\Omega, \mu)} \le C\sigma_v(f_0,\cD_N)_{L_\infty(\Og,\mu)},
\ee
where $C\ge 1$ is an absolute constant, $L_p^m$ denotes the $L_p$-space
 $L_p(\Omega_m,\mu_m)$, and $\{G_{k}^t(\cdot, \cD)_{L_p^m}\}_{k=1}^\infty$ denotes the
 WCGA defined  with respect to the  weakness parameter $t\in (0, 1]$  and  the system  $\cD_N(\Omega_m)$ in the space $L_p^m=L_p(\Omega_m,\mu_m)$. \end{Theorem}
 \begin{proof} Theorem \ref{CGT2} is a corollary of Theorem \ref{CGT1} and Proposition \ref{CGP1}. We begin with a proof of (\ref{mpc}).
 Consider separately two cases (I) $2\le p<\infty$ and (II) $1<p\le 2$.

 {\bf Case (I) $2\le p<\infty$.} By (\ref{CG2}) we have $\eta(L_p,w)\le (p-1)w^2/2,\  \ w>0.$ In our case
 $q=q'=2$. We set $c_\ast=C(t,p) := C(t,(p-1)/2,2)$, where $C(t,\gamma,q)$ is from Theorem \ref{CGT1}. Proposition \ref{CGP1} guarantees that we can apply Theorem \ref{CGT1} to the system
 $\cD_N(\Omega_m)$ in the space $L_p(\Omega_m,\mu_m)$. This gives us inequality (\ref{mpc}).

 {\bf Case (II) $1<p\le 2$.} By (\ref{CG2}) in this case we have $\eta(L_p,w)\le w^p/p,\  \ w>0.$ Also, $q=p$ and $q'=p/(p-1)=p'$ defined in Theorem \ref{CGT1}. We set $c=C(t,p) := C(t,1/p,p)$, where $C(t,\gamma,q)$ is from Theorem \ref{CGT1}. Proposition \ref{CGP1} guarantees that we can apply Theorem \ref{CGT1} to the system
 $\cD_N(\Omega_m)$ in the space $L_p(\Omega_m,\mu_m)$. This gives us inequality (\ref{mpc}).

We now derive (\ref{mp2c}) from (\ref{mpc}).   Clearly,
$$
\sigma_v(f_0,\cD_N(\Omega_m))_{L_p(\Omega_m,\mu_m)} \le \sigma_v(f_0,\cD_N )_\infty.
$$
Let $f\in \Sigma_v(\cD_N)$ be such that  $\|f_0-f\|_\infty = \sigma_v(f_0,\cD_N)_\infty$. Let us set $v':= \lceil c (2V)^{q'}(\ln (2Vv)) v^{rq'}\rceil$ for brevity. Then (\ref{mpc}) implies
$$
\|f - G^t_{v'}(f_0,\cD_N(\Omega_m))\|_{L_p(\Omega_m,\mu_m)} \le \|f-f_0\|_{L_p(\Omega_m,\mu_m)} +\|f_{v'}\|_{L_p(\Omega_m,\mu_m)}
$$
$$
\le (1+C)\sigma_v(f_0,\cD_N)_\infty.
$$
Using that $f - G_{v'}^t(f_0,\cD_N(\Omega_m)) \in \Sigma_u(\cD_N)$, by discretization (\ref{ud}) we
conclude that
\be\label{ub3}
\|f - G_{v'}^t(f_0,\cD_N(\Omega_m))\|_{L_p(\Omega,\mu)} \le 2^{1/p}(1+C)\sigma_v(f_0,\cD_N)_\infty.
\ee
Finally,
$$
\|f_{v'}\|_{L_p(\Omega,\mu)} \le \|f-f_0\|_{L_p(\Omega,\mu)} + \|f - G_{v'}^t(f_0,\cD_N(\Omega_m))\|_{L_p(\Omega,\mu)}.
$$
This and (\ref{ub3}) prove (\ref{mp2c}).

\end{proof}

\begin{Theorem}\label{CGT3}
	Under the conditions of Theorem \ref{CGT2}, we have
	\be\label{mp3}
	\Big\|f_0-G_{a(v)}^t(f_0, \cD)_{L_p^m}\Big\|_{L_p(\Omega, \mu)} \le C' \sigma_v(f_0,\cD_N)_{L_p(\Omega, \mu_\xi)},
	\ee
	where   $a(v)$ is    from Theorem \ref{CGT2}, $C'$ is a positive absolute constant, and
	$$
	\mu_\xi := \frac {\mu+\mu_m}{2}=\frac{1}{2} \mu + \frac{1}{2m} \sum_{j=1}^m \delta_{\xi^j}.
	$$
\end{Theorem}
\begin{proof}
	For convenience, we use the notation $\|\cdot\|_{p,\nu}$  to denote the norm of $L_p$ defined with respect to a measure $\nu$ on $\Omega$.
	Let $g\in \Sigma_v(\cD_N)$ be such that  $\|f_0-g\|_{p,\mu_\xi} = \sigma_v(f_0,\cD_N)_{p,\mu_\xi}$.
Let $$v':=a(v)=\lceil c(2V)^{q'} (\ln (2Vv)) v^{rq'}\rceil.$$
	Then
	\begin{align*}
	\|f_{v'} \|_{p,\mu}&\leq
	2^{1/p} \|f_0 - G_{v'}^t(f_0,\cD_N)_{L_p^m}\|_{p,\mu_\xi}\\
	&\leq 2^{1/p}\|f_0-g\|_{p,\mu_\xi}+2^{1/p} \|g- G_{v'}^t(f_0,\cD_N)_{L_p^m}\|_{p,\mu_\xi}\\
	&\leq 2^{1/p} \sigma_v(f_0,\cD_N)_{p,\mu_\xi}+ 2^{1/p}\|g- G_{v'}^t(f_0,\cD_N)_{L_p^m}\|_{p,\mu_\xi}.
	\end{align*}
	Since
	\[ g-G_{v'}^t(f_0,\cD_N)_{L_p^m}\in \Sigma_{v+v'} (\cD_N) \subset \Sigma_u(\cD_N),\]
	it follows by the $L_p$-usd assumption that
	\begin{align*}
	\|g-& G_{v'}^t(f_0,\cD_N)_{L_p^m}\|_{p,\mu_\xi}\leq C_1 \|g- G_{v'}^t(f_0,\cD_N)_{L_p^m}\|_{p,\mu_m}\\
	&	\leq C_1\|f_0-g\|_{p,\mu_m}+C_1 \|f_0- G_{v'}^t(f_0,\cD_N)_{L_p^m}\|_{p,\mu_m}\\
	&\leq 2^{1/p}C_1 \|f_0-g\|_{p,\mu_\xi}+ C_1 \|f_{v'} \|_{p,\mu_m},\end{align*}
	which, by Theorem \ref{CGT2}, is estimated by
	\begin{align*}
	&\leq C_2  \sigma_v(f_0,\cD_N)_{p,\mu_\xi}+ C_3\sigma_v(f_0,\cD_N(\Omega_m))_{p,\mu_m}\leq C' \sigma_v(f_0,\cD_N)_{p,\mu_\xi}.
	\end{align*}

\end{proof}

\section{Recovery on classes with structural condition}
\label{str}

 \subsection{Recovery on classes with conditions on the cubic decompositions}

For the reader's convenience we give a detailed definition of the classes $\bA^r_\bt(\Psi)$, which were introduced and studied in \cite{VT202}. For a given $1\le p\le \infty$,
let $\Psi =\{\psi_{\bk}\}_{\bk \in \bbZ^d}$  be a system in the space $L_p(\Omega,\mu)$ such that $\psi_\bk \in \cC(\Omega)$ and  $\|\psi_\bk\|_p \le B$ for all  $\bk\in\bbZ^d$. We consider functions representable in the form of an  absolutely convergent series
\be\label{Irepr}
f = \sum_{\bk\in\bbZ^d} a_\bk(f)\psi_\bk \quad\  \ \text{with}\  \  \sum_{\bk\in\bbZ^d} |a_\bk(f)|<\infty.
\ee
For $\bt \in (0,1]$ and $r>0$,  consider the class $\bA^r_\bt(\Psi)$ of all functions $f$ which   have representations (\ref{Irepr}) satisfying the  conditions
\be\label{IAr}
  \left(\sum_{\lfloor2^{j-1}\rfloor\le \|\bk\|_\infty <2^j} |a_\bk(f)|^\bt\right)^{1/\bt} \le 2^{-rj},\quad j=0,1,\dots  .
\ee

 The following Theorem \ref{RT1} was proved in \cite{VT202}.

  \begin{Theorem}[{\cite{VT202}}]\label{RT1} Let $1<p<\infty$, $r>0$, $\bt\in (0,1]$. Assume that $\sup_{\bk\in\bbZ^d}\|\psi_\bk\|_{L_p(\Omega,\mu)} \le B$ for some constant $B\ge 1$. Then there exist two positive constants $c^*=c(r,\bt,p,d)$ and $C=C(r,\bt,p,d)$ such that
 for any integer $v\in \bbN$,  there is an integer  $J\in\bbN$,  independent of the measure $\mu$, such that $2^J \le v^{c^*}$, and  the system $$\Psi_J := \Big\{\psi_\bk:\  \bk\in\bbZ^d,\  \|\bk\|_\infty <2^J\Big\},$$
 we have
 $$
 \sigma_v(\bA^r_\bt(\Psi),\Psi_J)_{L_p(\Omega,\mu)} \le CBv^{1/q-1/\bt-r/d},
 $$
 where  $q:= \min\{p,2\}$.
 Moreover, this bound is provided by a simple greedy algorithm.
 \end{Theorem}

 The following Theorem \ref{RT3a} was proved in \cite{VT203}.

\begin{Theorem}[{\cite{VT203}}]\label{RT3a}  Let $\Psi\subset L_2(\Omega,\mu)$ be  a uniformly bounded Riesz system satisfying (\ref{ub}) and (\ref{Riesz}) for some constants $0<R_1\leq R_2<\infty$. Let $v\in\mathbb{N}$, and  $u := \lceil (1+c_\ast)v\rceil$, where $c_\ast$ is the constant  from Theorem \ref{NUT1}.
Assume  that $\Psi \in NI(2,p,H,u)$ for some constants  $2\le p<\infty$ and  $H>0$.  Then for any $r>0$ and $\bt\in (0,1]$, there exist positive constants $c'=c'(r,\bt,p,R_1,R_2,d)$ and $C'=C'(r,\bt,p,d)$ such that
 \begin{equation}\label{R4s}
 \varrho_{m}^{o}(\bA^r_\bt(\Psi),L_p(\Omega,\mu)) \le C' H v^{1/2 -1/\bt -r/d}
\end{equation}
for any integer $m$ satisfying
$$
m\ge c'    v  (\log(2v ))^4 .
$$
Moreover, this bound is provided by the WOMP.
\end{Theorem}
\begin{proof} First, we apply Theorem \ref{RT1} in the space $L_p(\Omega,\mu)$ with  $B=1$, $2\le p<\infty$ and
$q=2$. As in Theorem
\ref{RT1}, consider the system $$\Psi_J := \Big\{\psi_\bk:\  \bk\in\bbZ^d,\ \ \|\bk\|_\infty <2^J \Big\}, $$
where $J\in\bbN$ does not depend on $\mu$, and satisfies  $2^J \le v^{c^*}$.
 By Theorem \ref{NUT1} with $\cD_N = \Psi_J$ and Theorem \ref{RT2} with $p=2$ and $K=R_1^{-2}$  there exist $m$ points  $\xi^1,\cdots, \xi^m\in  \Omega$  with
\be\label{mbound}
		m\leq C       u (\log N)^4\leq C u (\log v)^4,
\ee
such that  for any given $f_0\in \cC(\Omega)$,  the WOMP with weakness parameter $t$ applied to $f_0$ with respect to the system  $\cD_N(\Omega_m)$ in the space $L_2(\Omega_m,\mu_m)$ provides the following bound for
any integer $1\le v \le u/(1+c_\ast)$:
\be\label{R5}
	\|f_{\lceil{c_\ast v}\rceil} \|_{L_p(\Omega,\mu)} \le C'H \sigma_v(f_0,\cD_N)_{L_p(\Omega, \mu_\xi)}.
	\ee
In order to bound the right side of (\ref{R5}) we apply Theorem \ref{RT1} in the space $L_p(\Omega, \mu_\xi)$.
For that it is sufficient to check that $\|\psi_\bk\|_{L_p(\Omega, \mu_\xi)} \le 1$, $\bk\in\bbZ^d$. This follows from the assumption that $\Psi$ satisfies (\ref{ub}). Thus, by Theorem \ref{RT1},  we obtain that for $f_0 \in \bA^r_\bt(\Psi)$,
\be\label{R6}
 \sigma_v(f_0,\Psi_J)_{L_p(\Og, \mu_\xi)} \le Cv^{1/2 -1/\bt-r/d}.
 \ee
  Combining (\ref{R6}), (\ref{R5}), and taking into account (\ref{mbound}),
   we complete the proof.

\end{proof}

\begin{Corollary}[{\cite{VT203}}]\label{RC1}  Assume that $\Psi$ is a uniformly bounded Riesz system (in the space  $L_2(\Omega,\mu)$) satisfying (\ref{ub}) and (\ref{Riesz}) for some constants $0<R_1\leq R_2<\infty$.
Let $2\le p<\infty$  and $r>0$.
 There exist constants $c=c(r,\bt,p,R_1,R_2,d)$ and $C=C(r,\bt,p,d)$ such that   we have the bound
 \begin{equation}\label{R4w}
 \varrho_{m}^{o}(\bA^r_\bt(\Psi),L_p(\Omega,\mu)) \le C  v^{1-1/p -1/\bt -r/d}
\end{equation}
for any $m$ satisfying
$$
m\ge c   v  (\log(2v ))^4 .
$$
Moreover, this bound is provided by the WOMP.
\end{Corollary}

We derived Theorem \ref{RT3a} from Theorem \ref{NUT1}. We now formulate an analog of Theorem \ref{RT3a}, which can be derived from Theorem \ref{NUT3} in the same way as Theorem \ref{RT3a} has been derived from Theorem \ref{NUT1}.

\begin{Theorem}[{\cite{VT203}}]\label{RT4} Let $2\le p<\infty$ and let $m$, $v$ be given natural numbers.    Let $\Psi$ be  a uniformly bounded Bessel system satisfying (\ref{ub}) and (\ref{Bessel}) such that $\Psi \in NI(2,p,H,2v)$.
There exist constants $c=c(r,\bt,p,K,d)$ and $C=C(r,\bt,p,d)$ such that   we have the bound
 \begin{equation}\label{R4ssa}
 \varrho_{m}^{o}(\bA^r_\bt(\Psi),L_p(\Omega,\mu)) \le CH v^{1/2 -1/\bt -r/d}
\end{equation}
for any $m$ satisfying
$$
m\ge c   v  (\log(2v ))^4 .
$$
 \end{Theorem}

In the special case when $\Psi = \cT^d$ is the trigonometric system, Corollary \ref{RC1} gives
 \begin{equation}\label{R7}
 \varrho_{m}^{o}(\bA^r_\bt(\cT^d),L_p(\bbT^d)) \ll  v^{1-1/p -1/\bt -r/d} \quad\text{for}\quad  m\gg    v  (\log(2v ))^4.
\end{equation}

It is pointed out in \cite{DTM2} that known results on the RIP for the trigonometric system can be used
for improving results on the universal discretization in the $L_2$ norm in the case of the trigonometric system. We explain that in more detail.
 Let $M\in \bbN$ and $d\in \bbN$. Define $\Pi(M) := [-M,M]^d$ to be the $d$-dimensional cube. Consider the system $\cT^d(M) := \cT^d(\Pi(M))$ of functions $e^{i(\bk,\bx)}$, $\bk \in \Pi(M)$ defined on $\bbT^d =[0,2\pi)^d$. Then $\cT^d(M)$ is an orthonormal system in $L_2(\bbT^d,\mu)$ with $\mu$ being the normalized Lebesgue measure on $\bbT^d$. The cardinality of this system is $N(M):= |\cT^d(M)| = (2M+1)^d$. We are interested in bounds on $m(\cX_v(\cT^d(M)),2)$ in a special case, when $M\le v^c$ with some constant $c$, which may depend on $d$. Then Theorem \ref{RT2} with $p=2$
 gives
 \be\label{DT1}
 m(\cX_v(\cT^d(v^c)),2) \le C(c,d) v (\log (2v))^4.
 \ee
  It is  stated in \cite{DTM2} that the known results of \cite{HR} and \cite{Bour} allow us to improve the bound (\ref{DT1}):
 \be\label{HR1}
 m(\cX_v(\cT^d(v^c)),2) \le C(c,d) v (\log (2v))^3.
 \ee
 This, in turn, implies the following estimate
  \begin{equation}\label{R9}
 \varrho_{m}^{o}(\bA^r_\bt(\cT^d),L_p) \ll  v^{1-1/p -1/\bt -r/d} \quad\text{for}\quad  m\gg    v  (\log(2v ))^3.
\end{equation}

\subsection{Lower bounds}
\label{Lb}

Let us discuss lower bounds for the nonlinear characteristic $\varrho_m^o(\bA^r_\bt(\Psi),L_p)$.
We will do it in the special case when $\Psi$ is the trigonometric system $\cT^d := \{e^{i(\bk,\bx)}\}_{\bk\in \bbZ^d}$. For $\bN = (N_1,\dots,N_d)\in\mathbb{N}_0^d$, define
$$
\cT(\bN,d) := \left\{\sum_{\bk\in\mathbb{Z}^d:  -\bN\leq \bk\leq \bN} c_j e^{i(\bk,\bx)}:\  \ c_j\in\mathbb{C}\right\},\quad \vartheta(\bN) :=\prod_{j=1}^d (2N_j+1),
$$
where for $\mathbf{k}, \mathbf{n} \in \mathbb{Z}^d$, we write $\mathbf{k} \leq \mathbf{n}$ to mean that $k_j \leq n_j$ for all $1 \leq j \leq d$. Clearly,
\[\dim \cT(\bN,d)=\vartheta(\bN).\]
In this section $\Omega = \bbT^d$ and $\mu$ is the normalized Lebesgue measure on $\bbT^d$.

\begin{Lemma}[{\cite[Lemma 4.1]{VT203}}]\label{NLL1} Let $1\le q\le p\le \infty$ and let  $\cT(\bN,d)_q$ denote the unit $L_q$-ball of the subspace $\cT(\bN,d)$. Then  for any positive integer  $m\le \vartheta(\bN)/2$, we have
$$
\varrho_m^o(\cT(2\bN,d)_q,L_p) \ge c(d)\vartheta(\bN)^{1/q-1/p}  .
$$
\end{Lemma}

\begin{proof} Given a set  $\xi=\{\xi^1,\dots,\xi^m\}\subset \bbT^d := [0,2\pi)^d$  of $m$ points, we consider the subspace
$$
T(\xi) := \{f\in \cT(\bN,d):\, f(\xi^\nu) =0,\quad \nu=1,\dots,m\}.
$$
Let  $g_\xi\in T(\xi)$ be such that $|g_\xi(\bx^*)|=\|g_\xi\|_\infty=1$ for a point  $\bx^*$. For the further argument we need some classical trigonometric polynomials. Recall that the univariate Fej\'er kernel of order $j - 1$ is given by
$$
\mathcal K_{j} (x) := \sum_{|k|\le j} \bigl(1 - |k|/j\bigr) e^{ikx}
=\frac{(\sin (jx/2))^2}{j (\sin (x/2))^2}.
$$
The Fej\'er kernel is an even nonnegative trigonometric
polynomial of order $j-1$.  It satisfies the obvious relations
\be\label{FKm}
\| \mathcal K_{j} \|_1 = 1, \qquad \| \mathcal K_{j} \|_{\infty} = j.
\ee
Let $\cK_\bj(\bx):= \prod_{i=1}^d \cK_{j_i}(x_i)$ be the $d$-variate Fej\'er kernels for $\bj = (j_1,\dots,j_d)\in \mathbb{N}_0^d$ and $\bx=(x_1,\dots,x_d)$.
Define
\be\label{Rf}
f(\bx) := g_\xi(\bx)\cK_\bN(\bx-\bx^*),
\ee
where $\bN = (N_1,\dots,N_d)$.
 Then $f\in \cT(2\bN,d)$, $f(\xi^\nu)=0$, $\nu=1,\dots,m$, and
\be\label{R19}
\|f\|_q \le \|g_\xi\|_\infty \|\cK_\bN\|_q \le C_1(d)\vartheta(\bN)^{1-1/q}.
\ee
At the last step we used the known bound for the $L_q$ norm of the Fej\'er kernel (see \cite{VTbookMA}, p.83, (3.2.7)). By (\ref{Rf}) we get
\be\label{R20}
|f(\bx^*)| \ge C_2(d) \vartheta(\bN).
\ee
By the Nikol'skii inequality for the $\cT(2\bN,d)$ (see  \cite{VTbookMA}, p.90, Theorem 3.3.2),  we obtain
from (\ref{R20})
\be\label{R21}
\|f\|_p \ge C_3(d) \vartheta(\bN)^{1-1/p}.
\ee

Let $\cM$ be an arbitrary  mapping from $\bbC^m$ to $L_p$, and let  $g_0 := \cM(\mathbf 0)$. Set
$\tilde f:=f/\|f\|_q$.
Then
$$
\Big\|\tilde f-g_0\Big\|_p +\Big\|-\tilde f - g_0\Big\|_p \ge 2\Big\|\tilde f\Big\|_p\ge 2C(d) \vartheta(\bN)^{\frac 1q-\frac 1p},
$$
where the last step uses (\ref{R19}) and  (\ref{R21}). It follows that
\[ \max\Big\{ \|\tilde f-g_0\|_p, \|-\tilde f-g_0\|_p\Big\} \ge C(d) \vartheta(\bN)^{\frac 1q-\frac 1p}.\]
 Since  both $\pm \tilde f$  belong to $\cT(2\bN,d)_q$, and
\[ \cM \Big(  \tilde f(\xi^1), \cdots, \tilde f(\xi^m)\Big)= \cM \Big(  -\tilde f(\xi^1), \cdots, -\tilde f(\xi^m)\Big)=\cM(\mathbf 0)=g_0,\] this completes the proof of Lemma \ref{NLL1}.
\end{proof}

Lemma \ref{NLL1} implies the following lower bound for the classes $\bA^r_\bt(\cT^d)$.

\begin{Proposition}\label{NLP1a} For $\bt\in (0,1]$ and $r>0$ we have for $2\le p \le \infty$
\be\label{NL1b}
\varrho_m^o(\bA^r_\bt(\cT^d),L_p) \gg m^{1-1/p-1/\bt-r/d}.
\ee
\end{Proposition}
\begin{proof} Take $n\in\bbN$ and  set $N:=2^n-1$, $\bN := (N,\dots,N)$.
Then for $f\in \cT(\bN,d)_2$ we have by the H{\"o}lder inequality with parameter $2/\bt$
$$
  \left(\sum_{\bk: \|\bk\|_\infty\le N} |\hat f(\bk)|^\bt\right)^{1/\bt} \le (2N+1)^{d(1/\bt-1/2)} \left(\sum_{\bk: \|\bk\|_\infty\le N} |\hat f(\bk)|^2\right)^{1/2}
$$
$$
  =(2N+1)^{d(1/\bt-1/2)}\|f\|_2  .
$$
This bound and Lemma \ref{NLL1} with $q=2$ imply (\ref{NL1b}).
\end{proof}

We now show how Theorem \ref{NUT2} can be used to show  that Theorem \ref{IT2} cannot be substantially improved in the sense of relations between $m$ and $u$.

\begin{Proposition}[{\cite[Proposition 4.1]{VT203}}]\label{NLP1} Let $\bN=(N_1,\cdots, N_d)\in\bbN_0^d$ and define the trigonometric system
$$
\cD_N = \Big\{e^{i(\bk,\bx)}:\  \ \bk\in\bbZ^d,\   -2\bN\leq \bk\leq 2\bN  \Big\},
$$
where the cardinality of the system  is given by
\[ N=|\cD_N|=\vartheta(2\bN)=\prod_{j=1}^d (4N_j+1).\]
Let $p\in (2,\infty)$. Assume that there exists a set $\xi=\{\xi^1,\cdots, \xi^m\}\subset \bbT^d$ of $m\leq  \vartheta(\bN)/2$ points  that provides the one-sided $L_p(\bbT^d)$-universal sampling discretization (\ref{D6}) with  constant $D>0$  for  the collection  $\cX_{2v}(\cD_N)$ and some integer $1\leq v<N$. Then
$$
v \leq  C(d,p)(2D+1)^2 (\vartheta(\bN))^{2/p},
$$
where $C(d,p)>0$ is a constant depending only on $p$ and $d$.
\end{Proposition}

\begin{proof} By Theorem \ref{NUT2}, the  inequality (\ref{D7}) holds for any $f\in\cC(\bbT^d)$. We take the function $f$ from the proof of Lemma \ref{NLL1} defined in (\ref{Rf}). By the definition of $f$ we get $f(\xi^j)=0$ for all $j=1,\dots,m$. Therefore, $B_v(f,\cD_N,L_p(\xi))=0$ and by (\ref{R21}),  we obtain
\be\label{D9}
 \|f-B_v(f,\cD_N,L_p(\xi))\|_p \ge C_3(d)\vartheta(\bN)^{1-1/p}.
\ee
On the other hand,  using (\ref{R19}) with $q=2$,  we get
\be\label{D10}
\sum_{\bk\in \bbZ^d} |\hat f(\bk)| \le \vartheta(2\bN)^{1/2} \|f\|_2^{1/2}  \le C_1(d)\vartheta(\bN).
\ee
By (\ref{CG2}),  we obtain that $$\eta(L_p(\bbT^d,\mu_\xi),w) \le (p-1)w^2/2,\  \ \forall w>0.$$
It is known (see, for instance, \cite{VTbook}, p.342, and Lemma \ref{NLv} above) that for any dictionary $\cD = \{g\}$, $\|g\|_X \le 1$ in a Banach  space $X$ with $\eta(X,w) \le \gamma w^q$, $1<q\le 2$, we have
 \be\label{R1}
 \sigma_v(A_1(\cD),\cD)_X \le C(q,\gamma)(v+1)^{1/q -1}.
 \ee
 Here
 $$
 A_1(\cD) := \left\{f:\, f=\sum_{i=1}^\infty a_ig_i,\quad g_i\in \cD,\quad \sum_{i=1}^\infty |a_i| \le 1 \right\}.
 $$

We now apply the
inequality (\ref{R1}) and  (\ref{D10}) to obtain
\be\label{D11}
\sigma_v(f,\cD_N)_{L_p(\bbT^d, \mu_\xi)} \le C_1(d,p)\vartheta(\bN) (v+1)^{-1/2}.
\ee
Substituting (\ref{D9}) and (\ref{D11}) into (\ref{D7}) we find
$$
v+1 \le C(d,p)(2D+1)^2 \vartheta(\bN)^{2/p}
$$
with some positive constant $C(d,p)$. We choose this constant $C(d,p)$ and complete the proof.
\end{proof}

\subsection{Recovery on classes with conditions on the hyperbolic cross decompositions}

In this section, we expand upon the classes $\bW^{a,b}_{A_\beta}(\Psi)$ introduced in Subsection \ref{RIfc}, which  were initially defined  for the trigonometric system $\Psi = \cT^d$ with $\beta=1$.  We now give  the definition for a general system $\Psi=\{\psi_{\bk}\}_{\bk\in\bbZ^d}$ of functions on a probability space $(\Omega, \mu)$.

Let $\mathbf{s}=(s_1, \dots, s_d) \in \mathbb{N}_0^d$ be a vector of non-negative integers. We define the block of indices $\rho(\mathbf{s})$ as
$$\rho(\mathbf{s}) := \bigl\{ \mathbf{k} \in \mathbb{Z}^d : \lfloor 2^{s_j-1} \rfloor \le |k_j| < 2^{s_j}, \quad j=1, \dots, d \bigr\}.$$For a function $f$ representable by the absolutely convergent series\begin{equation}\label{R9a}f = \sum_{\mathbf{k}\in\bbZ^d} a_{\mathbf{k}}(f) \psi_{\mathbf{k}}   \quad \text{with}\  \  \sum_{\mathbf{k}\in\bbZ^d} |a_{\mathbf{k}}(f)| < \infty,\end{equation}we define  $$\delta_{\mathbf{s}}(f, \Psi) := \sum_{\mathbf{k} \in \rho(\mathbf{s})} a_{\mathbf{k}}(f) \psi_{\mathbf{k}}, \    \  \quad f_j := \sum_{\bs\in\bbN_0^d: \  \|\mathbf{s}\|_1 = j} \delta_{\mathbf{s}}(f, \Psi), \quad j \in \mathbb{N}_0,$$ and for  $\beta \in (0, 1]$,
 $$
 |f|_{A_\beta(\Psi)} := \left( \sum_{\mathbf{k}\in\bbZ^d} |a_{\mathbf{k}}(f)|^\beta \right)^{1/\beta}.
 $$
 For parameters $a \in \mathbb{R}_+$ and $b \in \mathbb{R}$, the class $\mathbf{W}^{a,b}_{A_\beta}(\Psi)$ consists of functions $f$ possessing a representation \eqref{R9a} such that the hyperbolic layers satisfy
 \begin{equation}\label{R10}
 |f_j|_{A_\beta(\Psi)} \le 2^{-aj}(\bar{j})^{(d-1)b}, \quad j \in \mathbb{N}_0,
 \end{equation}
 where  $ \bar{j} := \max(j, 1)$.

 These classes were originally introduced for the trigonometric system with $\beta=1$ in \cite{VT150}, with the general definition for any $\Psi$ and $\beta=1$  later appearing in \cite{DTM2}. For simplicity, we denote $\mathbf{W}^{a,b}_{A}(\Psi) := \mathbf{W}^{a,b}_{A_1}(\Psi)$.

The following result from \cite{VT205} establishes upper bounds for the optimal nonlinear recovery of the classes $\mathbf{W}^{a,b}_{A_\beta}(\Psi)$ in the $L_p$ norm.

\begin{Theorem}[{\cite{VT205}}]\label{MT1}  Assume that $\Psi$ is a uniformly bounded Riesz system in   $L_2(\Omega,\mu)$ satisfying  the uniform bound condition (\ref{ub}) and the Riesz condition (\ref{Riesz}) with constants $0<R_1\leq R_2<\infty$.
Let $2\le p<\infty$  and $a>0$.
 There exist positive constants $c=c(a,b,\bt,p,R_1,R_2,d)$ and $C=C(a,b,\bt,p,d)$ such that  the following bound holds:
 \begin{equation}\label{R4w'}
 \varrho_{m}^{o}(\bW^{a,b}_{A_\bt}(\Psi),L_p(\Omega,\mu)) \le C  v^{1-1/p -1/\bt -a}  (\log(2v))^{(d-1)(a+b)}
\end{equation}
for any number of samples  $m$ satisfying
$$
m\ge c   v  (\log(2v ))^4 .
$$
Furthermore, this recovery bound is constructively achieved by the Weak Orthogonal Matching Pursuit (WOMP) algorithm.
\end{Theorem}

In the specific case where the system $\Psi$ is the trigonometric system $\cT^d$, the logarithmic factor in the required number of samples $m$ can be slightly improved.

\begin{Theorem}[{\cite{VT205}}]\label{MT2}  Assume that $\Psi$ is the  $d$-variate trigonometric system $\cT^d$. Let $2\le p<\infty$  and $a>0$.
 There exist positive constants $c=c(a,b,\bt,p,d)$ and $C=C(a,b,\bt,p,d)$ such that the following recovery bound holds:
 \begin{equation}\label{Mtr}
 \varrho_{m}^{o}(\bW^{a,b}_{A_\bt}(\cT^d),L_p(\bbT^d)) \le C  v^{1-1/p -1/\bt -a}  (\log(2v))^{(d-1)(a+b)}
\end{equation}
for any number of samples $m$ satisfying
$$
m\ge c   v  (\log(2v ))^3 .
$$
Moreover, this bound is achieved constructively by the Weak Orthogonal Matching Pursuit (WOMP).
\end{Theorem}

Let $Q$ be a finite subset of  $\mathbb Z^d$. We  define the subspace $\Psi(Q)$ as
$$
\Psi(Q) :=\left\{ \sum_{\mathbf k\in Q}a_{\mathbf k}
\psi_\bk:\  \ a_{\mathbf k}\in\mathbb{C}   \right\} .
$$
The proof of Theorem \ref{MT1} is based on the following sparse approximation result.
\begin{Theorem}[{\cite{VT205}}]\label{NT0}   Let $1<p<\infty$, $a>0$ and  $b\in \bbR$. Assume that $\|\psi_\bk\|_{L_p(\Omega,\mu)} \le 1$ for all  $\bk\in\bbZ^d$. Then there exist  constants $c=c(a,\bt,d,p)$ and  $C=C(a,b,\bt,d,p)$, such that,  for any $v\in \bbN$,  there exists  a subset  $Q\subset \bbZ^d$ with  $|Q| \le v^{c}$ that is independent of the measure $\mu$, for which
there exists
a constructive method $A_{v,\mu}$,  based on greedy algorithms, that provides a $v$-term approximant
from $\Psi(Q)$  such that  for any  $f\in \bW^{a,b}_{A_\bt}(\Psi)$,
$$
\|f-A_{v,\mu}(f)\|_{L_p(\Omega,\mu)} \le C   v^{-a+1/q-1/\bt} (\log(2 v))^{(d-1)(a+b)},
$$
where $q:=\min(p,2)$.
\end{Theorem}

The following Lemma \ref{NL1} is a corollary of Theorem \ref{NT0}.

\begin{Lemma}[{\cite{VT205}}]\label{NL1}  Assume that $\Psi = \{\psi_{\mathbf{k}}\}_{\mathbf{k}\in\mathbb{Z}^d}$ is a uniformly bounded system of functions on $\Omega$ such that
\be\label{R8}
|\psi_\bk(\bx)| \le 1,\quad \forall \bx \in \Omega, \quad\forall \bk \in \bbZ^d.
\ee
Let $1<p<\infty$, $a>0$, $b\in \bbR$, and $v\in\bbN$. There exist two constants $c^*=c(a,\bt,d,p)$, $C=C(a,b,\bt,d,p)$, and a finite set $Q\subset \bbZ^d$ with cardinality $|Q| \le v^{c^*}$,
such that for any probability measure $\mu$ and any finite sample  set $\xi\subset \Omega$   there exists
a constructive method $A_{v,\xi}$ based on greedy algorithms  providing a $v$-term approximant
from the subspace $\Psi(Q)$  with the bound that   for $f\in \bW^{a,b}_{A_\bt}(\Psi)$,
$$
\|f-A_{v,\xi}(f)\|_{L_p(\Omega,\mu_\xi)} \le C   v^{-a+1/q-1/\bt} (\log(2 v))^{(d-1)(a+b)},
$$
where $q := \min(p, 2)$ and $\mu_\xi$ is the measure defined in \eqref{muxi}.

\end{Lemma}

The following Theorem \ref{RT3} is a variant of Theorem \ref{MT1} with an extra assumption on the Nikol'skii inequality.

 \begin{Theorem}[{\cite{VT205}}]\label{RT3}  Assume that $\Psi$ is a uniformly bounded Riesz system (in the space  $L_2(\Omega,\mu)$) satisfying (\ref{ub}) and (\ref{Riesz}) for some constants $0<R_1\leq R_2<\infty$.
Let $2\le p<\infty$  and $a>0$. For any $v\in\bbN$ denote $u := \lceil (1+c_\ast)v\rceil$, where $c_\ast$ is from Theorem \ref{NUT1}.  Assume in addition that $\Psi \in NI(2,p,H,u)$ with $p\in [2,\infty)$.
 There exist constants $c'=c'(a,b,\bt,p,R_1,R_2,d)$ and $C'=C'(a,b,\bt,p,d)$ such that
 \begin{equation}\label{R4sa}
 \varrho_{m}^{o}(\bW^{a,b}_{A_\bt}(\Psi),L_p(\Omega,\mu)) \le C' H v^{1/2 -1/\bt -a}  (\log(2v))^{(d-1)(a+b)}.
\end{equation}
for any $m$ satisfying
$$
m\ge c'    v  (\log(2v ))^4 .
$$
Moreover, this bound is provided by the WOMP.
\end{Theorem}

Note that Theorem \ref{RT3} implies Theorem \ref{MT1} (see \cite{VT205}). The following Theorem \ref{RT4a} provides bound (\ref{R4sa}) under weaker conditions on $\Psi$ with the help of other than WOMP algorithm.

\begin{Theorem}[{\cite{VT205}}]\label{RT4a} Let $2\le p<\infty$ and let $m$, $v$ be given natural numbers.    Let $\Psi$ be  a uniformly bounded Bessel system satisfying (\ref{ub}) and (\ref{Bessel}) such that $\Psi \in NI(2,p,H,2v)$.
There exist constants $c=c(a,b,\bt,p,K,d)$ and $C=C(a,b,\bt,p,d)$ such that
 \begin{equation}\label{R4ss}
 \varrho_{m}^{o}(\bW^{a,b}_{A_\bt}(\Psi),L_p(\Omega,\mu)) \le CH v^{1/2 -1/\bt -a}  (\log(2v))^{(d-1)(a+b)}
\end{equation}
for any $m$ satisfying
$$
m\ge c   v  (\log(2v ))^4 .
$$
 \end{Theorem}

 Similarly to the above discussion after Theorem \ref{RT4} in the special case $\Psi = \cT^d$ we can improve the restriction $m\ge c   v  (\log(2v ))^4$ to $m\ge c   v  (\log(2v ))^3$. This gives
 Theorem \ref{MT2}.

 \subsection{Recovery on general classes $\bA^{r,b}_\bt(\Psi,\cG)$}
 \label{rgc}

 In this subsection we discuss recent results from \cite{STT} on sampling recovery on general classes $\bA^{r,b}_\bt(\Psi,\cG)$. Note, that a special case of these classes, namely classes $\bA^r_\bt(\Psi,\cG)$, were defined in subsection \ref{RIfc}.  The following collection of classes was introduced
in \cite{VT203}. Given a system $\Psi = (\psi_\bk)_{\bk \in \bbZ^d}$,
consider a sequence of subsets
$\cG:=\{G_j\}_{j=1}^\infty$, $G_j \subset \bbZ^d$, such that
\be\label{rgcDi4}
  G_1\subset G_2\subset \cdots \subset G_j\subset G_{j+1}
  \subset \cdots,\quad \bigcup_{j=1}^\infty G_j =\bbZ^d,\quad G_0:=\emptyset\,,
\ee
and the functions representable in the form of absolutely convergent series
\be\label{rgcDirepr}
    f = \sum_{\bk\in\bbZ^d} a_\bk(f)\psi_\bk,
    \qquad \sum_{\bk\in\bbZ^d} |a_\bk(f)|<\infty.
\ee
For $\bt \in (0,1]$, $r>0$, and $b\in \bbR$
define the class $\bA^{r,b}_\bt(\Psi,\cG)$ as the functions $f$
which have representations (\ref{rgcDirepr}) satisfying the following
conditions
\be\label{rgcDiAr}
  \Big(\sum_{  \bk\in G_j\setminus G_{j-1}} |a_\bk(f)|^\bt\Big)^{1/\bt}
\le 2^{-rj} j^b,\quad j=1,2,\dots.
\ee

Two special cases of these classes have already been discussed above in this section.
In \cite{VT202} and \cite{VT203},
the behavior of $\varrho_m^o(\bF,L_p)$ for the classes
$\bA^r_\bt(\Psi):= \bA^{r,0}_\bt(\Psi,\cG^c)$ was studied, with
$\cG^c$ being dyadic cubes,
$$
   G^c_j := \{\bk: \,  \|\bk\|_\infty <2^j\}, \quad j \in \bbN\,.
$$

The second special case, for the classes
$\bW^{a,b}_{A_\bt}(\Psi):= \bA^{a,b}_\bt(\Psi,\cG^h)$
with $\cG^h$ being dyadic hyperbolic crosses, was
studied in \cite{VT205}.
In that case, with $\mathbf s=(s_1,\dots,s_d ) \in \bbN^d$,
the sets $\cG^h$ of hyperbolic crosses are defined as follows
$$
   \rho(\mathbf s) := \bigl\{ \mathbf k\in\mathbb Z^d:[ 2^{s_j-1}]
\le |k_j| < 2^{s_j},\quad j=1,\dots,d \bigr\},
   \quad G^h_j := \bigcup_{\|\bs\|_1 \le j}\rho(\mathbf s).
$$

The classes $\bA^r_\bt(\Psi)$,
with the sequence $ \cG^c:=\{G_j^c\}_{j=1}^\infty$ of dyadic cubes
(the index $c$ stands for {\it cubes}), are
related to the classical Sobolev-Nikol'skii-Besov classes
(for approximation in these classes see, for instance, \cite{VTbookMA}, Ch.3).
The classes  $\bW^{a,b}_{A_\bt}(\Psi)$, with the sequence
$ \cG^h:=\{G_j^h\}_{j=1}^\infty$ of dyadic hyperbolic crosses
(the index $h$ stands for {\it hyperbolic}) are related to the classical
classes with mixed smoothness
(for approximation in these classes see, for instance,
\cite{VTbookMA}, Ch.4--7).

It is well known that the study of the Sobolev-Nikol'skii-Besov classes
and classes with mixed smoothness requires different techniques.
Surprisingly, papers \cite{VT203} and \cite{VT205} showed that we can
use the same techniques in the study of both $\bA^r_\bt(\Psi)$ and
$\bW^{a,b}_{A_\bt}(\Psi)$ classes. This motivated us to conduct the
corresponding research in the general setting of
classes $ \bA^{r,b}_\bt(\Psi,\cG)$.
It turns out that the technique developed in \cite{VT203} and \cite{VT205}
allows us to treat the classes $ \bA^{r,b}_\bt(\Psi,\cG)$ under a minor
restriction on the sequence of index sets $\cG$.
This minor restriction is the following condition on
the system $\cG$.

\medskip
{\bf Regularity condition RC($\th,\theta'$).}
We say that a system $\cG$ satisfies regularity condition
{\bf RC($\th,\theta'$)}
if there exist positive
numbers $\theta > \theta' > 0$ such that for any $j \in \bbN$
we have
\be\label{sa11}
   2^{\theta'} \le \frac{|G_{j+1}|}{|G_j|} \le 2^{\theta}\,,
   \quad j \in \bbN\,, \qquad |G_0| = 1\,.
\ee
This condition means that cardinalities of the index sets $\{G_j\}_{j\in\bbN}$
grow exponentially,
with a possibly varying rate which is however bounded
from below and from above.

\medskip
For a given $v \in \bbN$, define $n:=n(v)\in\bbN$ through
the following inequalities.
\be\label{sa2}
   |G_{n-1}| < v \le |G_n|.
\ee
Then, under the regularity condition
on the growth of cardinalities of $\{G_j\}$ as above,
  the following results were proved in \cite{STT}.
(See \cite{STT}
about a slightly more general condition under which these results are
still true.)

 The first result   shows the rate of nonlinear
$v$-term approximation of the classes $\bA^{r,b}_\bt(\Psi,\cG)$.

\begin{Theorem}\label{InT2}
Let $1<p<\infty$, $p^*:= \min\{p,2\}$, $r>0$, $b\in \bbR$, $\bt\in (0,1]$.
Assume that

1) $\|\psi_\bk\|_{L_p(\Omega,\mu)} \le B$, $\bk\in\bbZ^d$,
with some probability measure $\mu$ on $\Omega$,

2) the system $\cG$ satisfies condition {\bf RC($\th,\th'$)}.

\noindent
Then there exist two constants $c=c(r,b,\bt,p,\th,\th')$
and $C=C(r,b,\bt,p,\th, B)$ such that
for any $v\in \bbN$ there is a $J\in\bbN$, $|G_J| \le v^c$
with the property 
$$
   \sigma_v(\bA^{r,b}_\bt(\Psi,\cG),\Psi_J)_{L_p(\Omega,\mu)}
\le Cv^{1/p^* -1/\bt}2^{-rn(v)}n(v)^b,
    \quad \Psi_J := \{\psi_\bk\}_{\bk\in G_J}.
$$
Moreover, this bound is provided by a simple greedy algorithm.
\end{Theorem}

The second result   gives the rate of the optimal
sampling recovery for these classes.

\begin{Theorem}\label{InT1}
Let $2\le p<\infty$ and let $m$, $v$ be given natural numbers.

1) Let $\Psi$ be  a uniformly bounded Bessel system satisfying (\ref{ub}) and (\ref{Bessel}).

2) Assume in addition that the system $\cG$ satisfies condition {\bf RC}($\th,\th'$).

\noindent
Then there exist constants $c=c(r,\bt,p,K,\th,\th')$ and $C=C(r,\bt,p,\th,\th')$ such that   we have the bound
 \begin{equation}\label{rgcInt1}
 \varrho_{m}^{o}(\bA^{r,b}_\bt(\Psi,\cG),L_p(\Omega,\mu)) \le C v^{1-1/p -1/\bt}   2^{-rn(v)}n(v)^b
\end{equation}
for any $m$ satisfying
$$
m\ge c   v  (\log(2v ))^4 .
$$
\end{Theorem}

 Note, that under a stronger restriction on the system $\Psi$, when we require that $\Psi$ is  a uniformly bounded Riesz system satisfying (\ref{ub}) and (\ref{Riesz}), we can somewhat improve the bound for $m$ to $m\ge c   v  (\log(2v ))^3$
   and guarantee that the estimate (\ref{rgcInt1}) is provided by a greedy-type algorithm.

If $\theta=\theta'=d$, then
$$
   n(v) = \Big\lfloor\frac{\log v}{d}\Big\rfloor,
   \quad 2^{-rn(v)} \asymp v^{-r/d}
$$
and the estimate in Theorem \ref{InT2} becomes
$$
   \sigma_v(\bA^{r,b}_\bt(\Psi,\cG),\Psi_J)_{L_p(\Omega,\mu)}
\le Cv^{1/p^* -1/\bt - r/d} \ln^b v,
    \qquad \Psi_J := \{\psi_\bk\}_{\bk\in G_J}\,.
$$

\section{Recovery on classes with smoothness condition}
\label{sm}

We begin with the definition of classes $\bW^r_q$ and very brief comments on the univariate case.
In the univariate case, for $r>0$, let
\be\label{sr7}
F_r(x):= 1+2\sum_{k=1}^\infty k^{-r}\cos (kx-r\pi/2)
\ee
and in the multivariate case, for $\bx=(x_1,\dots,x_d)$, let
$$
F_r(\bx) := \prod_{j=1}^d F_r(x_j).
$$
Denote
$$
\bW^r_q := \{f:f=\varphi\ast F_r,\quad \|\varphi\|_q \le 1\},
$$
where
$$
(\varphi \ast F_r)(\bx):= (2\pi)^{-d}\int_{\bbT^d} \ff(\by)F_r(\bx-\by)d\by.
$$
The classes $\bW^r_q$ are classical classes of functions with {\it dominated mixed derivative}
(Sobolev-type classes of functions with mixed smoothness).
The reader can find results on approximation properties of these classes in the books \cite{VTbookMA} and \cite{DTU}.

Let $I_n$ and $R_n$ be the recovery operators defined in the beginning of Section \ref{RI} and $W^r_q$ stands for $\bW^r_q$ in the univariate case.

\begin{Theorem}[{\cite{VT51}}]\label{univarR} Let  $1\le q,p \le \infty$ and $r>1/q$. Then
$$
 \varrho_{4m}(W^r_q,L_p) \asymp \sup_{f\in W^r_q}\|f-R_m(f)\|_p \asymp m^{-r+(1/q-1/p)_+}.
$$
In the case  $1<p<\infty$ the above estimates are valid for the operator  $I_m$ instead of the operator  $R_m$.
\end{Theorem}

\subsection{Nonlinear recovery}
\label{sms1}

The authors of \cite{JUV} (see Corollary 4.16 in v3) proved the following bound for $1<q<2$, $r>1/q$ and  $m\ge c(r,d,q) v(\log(2v))^3$:
\begin{equation}\label{H6}
\varrho_{m}^{o}(\bW^{r}_q,L_2(\bbT^d)) \le C(r,d,q)  v^{-r+1/q-1/2} (\log v)^{(d-1)(r+1-2/q)+1/2}.
\end{equation}
The authors of  \cite{DTM2} proved the following bound
\be\label{H7}
\varrho_{m}^{o}(\bW^{r}_q,L_2(\bbT^d)) \le C'(r,d,q)  v^{-r+1/q-1/2} (\log v)^{(d-1)(r+1-2/q)}
\end{equation}
provided that
\be\label{H8}
m\ge c'(r,d,q) v(\log(2v))^3.
\ee

In the above mentioned results the sampling recovery in the $L_2$ norm has been studied. The technique, which was used in the proofs of the bounds (\ref{H6}) and (\ref{H7}),  is heavily based on the fact that we approximate in the $L_2$ norm.  The following bounds
have been proved in \cite{KoTe1}.

 \begin{Theorem}[{\cite{KoTe1}}]\label{HT2}  Let $1<q\le 2\le p<\infty$.  There exist two constants $c=c(r,d,p,q)$ and $C=C(r,d,p,q)$ such that     for $r>1/q$,
 \begin{equation}\label{H9}
  \varrho_{m}^{o}(\bW^{r}_q,L_p(\bbT^d)) \le C   v^{-r+1/q-1/p} (\log v)^{(d-1)(r+1-2/q)}
\end{equation}
   for any $v\in\bbN$ and any $m$ satisfying
$$
m\ge c  v(\log(2v))^3.
$$
\end{Theorem}

First, we discuss the above results and then explain how Theorem \ref{HT2} can be derived from results of Section \ref{str}.

Theorem \ref{HT2} extends existing bounds for $p=2$ to the broader range $p \in [2, \infty)$.
Let us make some comments on bounds (\ref{H6}) and (\ref{H7}). First of all, these bounds were proved by different methods. Bound (\ref{H6})
was proved in \cite{JUV} utilizing powerful results from compressed sensing theory, whereas  bound (\ref{H7}) is based on the  universal discretization of the $L_2$ norm, with a  proof centered on a greedy-type algorithm-- the Weak Orthogonal Matching Pursuit (WOMP).  Second, the WOMP produces an approximant that is sparse with respect to the trigonometric system, whereas the method in (\ref{H6}) does not guarantee such sparsity. Third, The algorithm in \cite{JUV} requires specific information about the function class (in this case, $\bW^{r}_q$), necessitating an algorithm tailored to each specific class. Conversely, the WOMP does not require prior knowledge of the function class to achieve bound (\ref{H7}), making it universal for the collection of classes $\bW^{r}_q$.
Finally, bound (\ref{H7}) is slightly sharper than (\ref{H6}).

The paper \cite{DTM3} is devoted to generalization of results of \cite{DTM2}, where recovery
in the norm $L_2$ was studied, to the case of $L_p$ spaces with $p\in [2,\infty)$. Instead of
universal discretization of the $L_2$ norm the universal discretization of the $L_p$ norm was used. The Weak Orthogonal Matching Pursuit was replaced by its generalization to the case of Banach spaces -- the Weak Chebyshev Greedy Algorithm (WCGA). Also, in addition to the WCGA a recovery algorithm based on the $\ell_p$ minimization (see $B_v(\cdot,\cD_N,L_p(\xi))$ above) was defined and studied. However, application of results from \cite{DTM3} to the sampling recovery of classes $\bW^{r}_q$ in the $L_p$ norm with $p\in (2,\infty)$ gives weaker than in Theorem \ref{HT2} bounds. It is related to the use of universal discretization of the $L_p$ norm.

The authors of \cite{KoTe1} used the recovery algorithm $B_v(\cdot,\cD_N,L_2(\xi))$, universal discretization of the $L_2$ norm, and the $u$-term Nikol'skii inequality.
On one hand the proof of Theorem \ref{HT2} goes along the lines of the corresponding proof
in \cite{DTM2}. On the other hand in the proof of Theorem \ref{HT2} they used a new concept
of the $u$-term Nikol'skii inequality, known deep results on sparse approximation in the
$L_p$ norm for $p\in [2,\infty)$, and a different recovery algorithm.

Denote
$$
\bW^{a,b}_A:= \bW^{a,b}_{A_1}(\cT^d).
$$
We now present some direct corollaries of Theorem \ref{MT2} with $\bt=1$  for other classes, which are better known
 than classes $\bW^{a,b}_A$. The following classes were introduced and studied in \cite{VT152} (see also \cite{VTbookMA}, p.364)
$$
\bW^{a,b}_q:=\{f: \|f_j\|_q \le 2^{-aj}(\bar j)^{(d-1)b},\quad \bar j:=\max(j,1),\quad j\in \bbN_0\}.
$$
The known results (see \cite{VTmon}, Ch.1, Theorem 2.2 and \cite{VTbookMA}, Theorem 9.1.10, p.454) imply that for $1<q\le 2$
\be\label{sr5}
\|f_j\|_A \le C(d,q)2^{j/q} j^{(d-1)(1-1/q)} \|f_j\|_q .
\ee
The inequality (\ref{sr5}) implies that for $1< q \le 2$ the class $\bW^{a,b}_q$ is embedded into
$\bW^{a',b'}_A$ with $a'= a-1/q$, $b'= b+1-1/q$. Therefore, Theorem \ref{MT2} implies the
following bounds for the classes $\bW^{a,b}_q$.

\begin{Theorem}[{\cite{KoTe1}}]\label{srT3}  Let $1<q\le 2\le p<\infty$.  There exist two constants $c=c(a,d,p,q)$ and $C=C(a,b,d,p,q)$ such that,
for $a>1/q$, we have the bound
 \begin{equation}\label{sr6}
  \varrho_{m}^{o}(\bW^{a,b}_q,L_p(\bbT^d)) \le C   v^{-a+1/q-1/p} (\log v)^{(d-1)(a+b+1-2/q)}
\end{equation}
for any $v\in\bbN$ and any $m$ satisfying
$$
m\ge c  v(\log(2v))^3.
$$
\end{Theorem}

Finally, let us consider the classical classes $\bW^r_q$, that have already been defined above.
It is well known (see, for instance, \cite{VTmon}, Ch.2, Theorem 2.1 and \cite{VTbookMA}, Theorem 4.4.9, p.174)
that for $f\in \bW^r_q$, $1<q<\infty$, one has
\be\label{sr8}
\|f_j\|_q \le C(d,q,r)2^{-jr},\quad j\in \bbN.
\ee
In particular, bounds (\ref{sr8}) imply that the class $\bW^r_q$ is embedded in the class $\bW^{r,0}_q$ for $1<q<\infty$.  Therefore, Theorem \ref{srT3} implies Theorem \ref{HT2}.

{\bf Some lower bounds.} The classes $\bW^r_q$ are classical classes of functions with {\it dominated mixed derivative}
(Sobolev-type classes of functions with mixed smoothness).
The reader can find results on approximation properties of these classes in the books \cite{VTbookMA} and \cite{DTU}.

 We now proceed to the definition of the classes $\bH^r_p$, which is based on the mixed differences. In this subsection we discuss results for these classes.

\begin{Definition}\label{InD2}
Let  $\btt =(t_1,\dots,t_d )$ and $\Delta_{\btt}^l f(\bx)$
be the mixed $l$-th difference with
step $t_j$ in the variable $x_j$, that is
$$
\Delta_{\btt}^l f(\bx) :=\Delta_{t_d,d}^l\dots\Delta_{t_1,1}^l
f(x_1,\dots ,x_d ) .
$$
Let $e$ be a subset of natural numbers in $[1,d ]$. We denote
$$
\Delta_{\btt}^l (e) :=\prod_{j\in e}\Delta_{t_j,j}^l,\qquad
\Delta_{\btt}^l (\varnothing) := Id \,-\, \text{identity operator}.
$$
We define the class $\bH_{q,l}^r B$, $l > r$, as the set of
$f\in L_q$ such that for any $e$
\be\label{B7a}
\bigl\|\Delta_{\btt}^l(e)f(\bx)\bigr\|_q\le B
\prod_{j\in e} |t_j |^r .
\ee
In the case $B=1$ we omit it. It is known (see, for instance, \cite{VTbookMA}, p.137) that the classes $\bH^r_{q,l}$ with different $l>r$ are equivalent. So, for convenience, we fix one $l= [r]+1$ and omit $l$ from the notation.
\end{Definition}

It is well known that in the univariate case ($d=1$) the approximation properties of the above $\bW^r_q$ and $\bH^r_q$ classes
are similar. It is also well known that in the multivariate case ($d\ge 2$) asymptotic characteristics (for instance, Kolmogorov widths, entropy numbers, best hyperbolic cross trigonometric approximations and others) have different rate of decay in
the majority of cases. Recently, a new scale of classes has been introduced and studied (see classes $\mathbf{W}^{a,b}_q$ and $\mathbf{W}^{a,b}_{A_\beta}(\Psi)$ defined above). It turns out that this scale is
convenient for simultaneous analysis of optimal sampling recovery of both the $\bW^r_q$ and the $\bH^r_q$ classes. We discuss here the case, when $\Psi =\cT^d$, and remind that   $\bW^{a,b}_A:= \bW^{a,b}_{A_1}(\cT^d)$. The following embedding result  easily  follows   from known results (see \cite{VT210}).

\begin{Proposition}\label{InP1} We have for $r>1/q$
\be\label{In1}
\bW^r_q \hookrightarrow \bW^{a,b}_A \quad \text{with} \quad a=r-1/q,\, b=1-1/q,\quad 1<q\le 2;
\ee
\be\label{In2}
\bH^r_q \hookrightarrow \bW^{a,b}_A \quad \text{with} \quad a=r-1/q,\, b=1,\quad 1\le q\le 2.
\ee
\end{Proposition}

In \cite{KoTe1} Theorem \ref{HT2} was derived from the embedding (\ref{In1}) and the following result for the $\bW^{a,b}_A$ classes (see \cite{KoTe1}, Theorem 3.3, Remark 3.1, and Proposition 3.1).

\begin{Theorem}[{\cite{KoTe1}}]\label{srT2}  Let $p\in [2,\infty)$. There exist two constants $c(a,p,d)$ and $C(a,b,p,d)$
such that we have the bound
\begin{equation}\label{sr3}
 \varrho_{m}^{o}(\bW^{a,b}_A,L_p(\bbT^d)) \le C(a,b,p,d)  v^{-a-1/p} (\log (2v))^{(d-1)(a+b)}
\end{equation}
    for any $v\in\bbN$ and any $m$ satisfying
$$
m\ge c(a,d,p) v(\log(2v))^3.
$$
\end{Theorem}

Theorem \ref{srT2} and embedding (\ref{In2}) imply the following analog of the bound (\ref{H9}) for the $\bH^r_q$ classes. There exist two constants $c=c(r,d,p,q)$ and $C=C(r,d,p,q)$ such that     we have the bound for $r>1/q$
 \begin{equation}\label{In6}
  \varrho_{m}^{o}(\bH^{r}_q,L_p(\bbT^d)) \le C   v^{-r+1/q-1/p} (\log (2v))^{(d-1)(r+1-1/q)}
\end{equation}
   for any $v\in\bbN$ and any $m$ satisfying
$$
m\ge c  v(\log(2v))^3.
$$
However, we point out that the bound (\ref{In6}) is weaker than the corresponding known bound for the linear recovery
(see section Discussion in \cite{KoTe1}).

The following lower bound for the $\bH^r_q$ classes is the main result of \cite{KoTe1}.

\begin{Theorem}\label{qpT1} For  $1\le q\le p< \infty$, $p>1$, $r>1/q$, we have
$$
\varrho_m^o(\bH^r_q,L_p) \ge c(d)m^{-r +1/q-1/p}  (\log m)^{(d-1)/p}  .
$$
\end{Theorem}

  Note that a new nontrivial feature of Theorem \ref{qpT1} is the logarithmic factor $(\log m)^{(d-1)/p}$, which shows that some logarithmic in $m$ factor is needed.

In \cite{KoTe1}   the following lower bound was derived from the known results developed in numerical integration.

\begin{Proposition}\label{q1P2} We have for $r>0$
$$
\varrho_m^o(\bH^r_\infty,L_1) \gg m^{-r}(\log m)^{d-1}.
$$
\end{Proposition}

\subsection{Linear recovery}
\label{sms2}

We begin with brief comments on the classical Smolyak recovery operators. For a detailed discussion of these and related operators, we refer the reader  to the books \cite{VTbookNS}, \cite{VTbookMA}, and \cite{DTU}. Let $\{R_n\}_{n=1}^\infty$ be the sequence of sampling recovery operators defined in \eqref{rco}. For each $i=1, \dots, d$, let $R_n^i$ denote the operator $R_n$ acting specifically on the variable $x_i$. We define the univariate difference operator as
$$
 \Delta_s^i := R_{2^s}^i - R_{2^{s-1}}^i,\  \ s\in \bbN_0,\  \ \quad R_{1/2} =0,
$$
and for a multi-index $\mathbf{s} = (s_1, \dots, s_d) \in \mathbb{N}_0^d$, we define the multivariate operator
$$
 \Delta_\bs := \prod_{i=1}^d \Delta_{s_i}^i.
$$
The corresponding Smolyak operator is defined as
$$
 T_n := \sum_{\bs\in\bbN_0^d:\  \|\bs\|_1\le n} \Delta_\bs.
$$
The operator $T_n$ utilizes $m$ function values, where the number of samples satisfies $$m \ll \sum_{k=1}^n 2^k k^{d-1} \ll 2^n n^{d-1}.$$

In 1960, S. Smolyak established the following error bound in the uniform norm:
$$
 \sup_{f\in\bW^r_\infty} \|f-T_n f\|_\infty \ll 2^{-rn}n^{d-1},\quad r>0.
$$
This result was subsequently extended to $L_p$ norms with $p < \infty$ in  \cite{VT23}:
$$
 \sup_{f\in\bW^r_p} \|f-T_nf\|_p \ll 2^{-rn}n^{d-1},\quad r>1/p.
$$
Further developments include the following results from \cite{VT51}:
$$
 \varrho_m(\bW^r_2)_\infty \asymp m^{-r+1/2} (\log m)^{r(d-1)},\quad r>1/2.
$$
In this case, the order of optimal recovery is attained by the Smolyak operator $T_n$ with an appropriately chosen $n$.
Additionally, for $q \in (1, \infty)$ and $r>\frac 1q$, we have (\cite{VT51})
$$
  \sup_{f\in \bW^r_q}\|f-T_n(f)\|_\infty \asymp 2^{-(r-1/q)n}n^{(d-1)(1-1/q)}.
$$

In the preceding Subsections, the discussion focused on the study of the nonlinear characteristic $\varrho_m^o(W, L_p)$. However, the majority of established results in optimal sampling recovery deal with  linear recovery methods. We now provide brief comments on recent advancements in this area and refer the reader to  the books \cite{DTU}, \cite{VTbookMA} and  the survey paper \cite{KKLT}  for a comprehensive discussion of previous results in this direction.

For special sets $\bF$ in the setting of reproducing kernel Hilbert spaces,  the following inequality was established  (see \cite{DKU}, \cite{NSU}, \cite{KU}, and \cite{KU2}):
\be\label{H5a}
\ro_{cn}(\bF,L_2) \le \left(\frac{1}{n}\sum_{k\ge n} d_k (\bF, L_2)^2\right)^{1/2},
\ee
where $c>0$ is  an absolute constant.  Here, $d_k (\bF, L_2)$ denotes  the Kolmogorov $k$-width of $\bF$ in
the space $L_2$.

Established results for the Kolmogorov widths $d_k (\bW^r_p, L_2)$ (see, for instance, \cite{VTbookMA}, p.216) asserts that for  $1<q\le 2$ with  $r > 1/q-1/2$ and for  $2<q<\infty$ with $r> 0$, the following asymptotic holds:
\be\label{H6a}
d_k(\bW^r_q,L_2) \asymp \left(\frac{(\log k)^{d-1})}{k}\right)^{r-(1/q-1/2)_+},
\ee
where $(a)_+ := \max(a,0)$.
When combined with \eqref{H5a}, this yields the following upper bounds for $1 < q \le \infty$ and $r > \max(1/q, 1/2)$:
$$
 \ro_m(\bW^r_q,L_2) \ll  \left(\frac{(\log m)^{d-1})}{m}\right)^{r-(1/q-1/2)_+}.
$$
These bounds solve the problem on  the correct decay orders for the sequence $\rho_m(\mathbf{W}^r_q, L_2)$ in the case when  $1 < q < \infty$ and $r > \max(1/q, 1/2)$, as it is evident that  $ \ro_m(\bW^r_q,L_2) \ge  d_m(\bW^r_q,L_2)$. It was observed in \cite{JUV} that for $q<2$ and sufficiently large $d$,  the upper bound (\ref{H6}) and the known lower bound from (\ref{H6a}) imply that $$\ro_m^o(\bW^r_q,L_2) =o(\ro_m(\bW^r_q,L_2)).$$
This demonstrates that in those cases, nonlinear methods achieve a better  rate of sampling recovery than  linear methods.

In the  case of our interest where $1 < q < 2 < p < \infty$, the correct order of $\ro_m(\bW^r_q,L_p)$ remains unknown. Known results utilizing Smolyak point sets provide the following upper bound for $1 < q < p < \infty$ and $r > 1/q$ (see, for instance, \cite{VTbookMA}, p. 308):
\be\label{Sm}
 \ro_m(\bW^r_q,L_p) \ll  \left(\frac{(\log m)^{d-1})}{m}\right)^{r-1/q+1/p}(\log m)^{(d-1)/p}.
 \ee
A comparison between \eqref{Sm} and the bounds established in Theorem \ref{HT2} reveals that for certain parameters $d$ and $p$, nonlinear recovery methods provide a better error rate than the linear recovery bound given by \eqref{Sm}.

The above inequality (\ref{H5a}) can only be used when the series on its right side converges. In cases where the corresponding series diverges,
the following  alternative inequality established in  \cite{VT183} (see also Theorem \ref{VT183T2} above) becomes particularly useful:  There exist two positive absolute constants $b$ and $B$ such that for any compact subset $W$ of $\cC(\Omega)$, we have
$$
\ro_{bn}(W,L_2) \le Bd_n(W,L_\infty).
$$
This inequality, combined with the upper bounds for the Kolmogorov widths obtained in \cite{TU1},
gives the following bounds for sampling recovery (see \cite{TU1}): for $2<q\le\infty$ and $1/q<r<1/2$,
\be\label{H3}
 \ro_m(\bW^r_q,L_2) \ll m^{-r} (\log m)^{(d-2)(1-r)+1}.
\ee

\section{Concluding remarks}
\label{Con}

Sampling recovery is an actively developing area of research. It has its roots in the classical problem of interpolation of functions. Traditionally, this area of research belongs to approximation theory and harmonic analysis. Recent developments in this area show its close connections with other areas of research. We demonstrate in this survey that important progress in sampling recovery was achieved by using recent deep breakthrough results in discretization and sparse approximation in Banach spaces. In turn, discretization results use fundamental results from finite dimensional geometry obtained in the papers \cite{BSS} and \cite{MSS}. Sparse approximation is the central topic in compressed sensing and in greedy approximation (see the books \cite{FR}, \cite{VTbook}, and \cite{VTbookMA}).

This survey complements the existing surveys on discretization and sampling recovery \cite{DPTT}, \cite{KKLT}, \cite{LMT}, and \cite{DTdiscsurv}. The area is actively developing and we were not able to discuss in detail all recent results on the topic. We tried to focus on a systematic presentation of new results directly related to results obtained in our group. We did not include open problems in this survey.
We refer the reader to the surveys \cite{DPTT}, \cite{KKLT}, \cite{LMT}, where a number of open problems are formulated.
Recently, there were important steps made in sampling recovery (see Steps 1-4 mentioned at the end of Section \ref{RI})
and we discussed them in detail here. The recent progress in sampling recovery shows the direction of its development -- from special problems to more and more general settings and problems. For instance, this includes steps from the univariate smoothness classes to the multivariate smoothness classes and to the multivariate classes with structural conditions.

  \Addresses

 \end{document}